\documentclass[11pt,reqno]{amsart}
\usepackage{xifthen}
\usepackage{subfig}
\usepackage{amssymb,amsfonts,amsmath,amsthm,amscd,dsfont,mathrsfs,mathtools,bm,comment,enumerate}
\usepackage{graphicx,float,psfrag,epsfig}
\usepackage{cleveref}
\usepackage{wrapfig}
\usepackage[utf8]{inputenc}
\usepackage[T1]{fontenc}
\usepackage[toc]{appendix}
\usepackage[usenames,dvipsnames]{color}

\DeclareMathAlphabet{\mathpzc}{OT1}{pzc}{m}{it}

\newcommand{\edges}{\mathsf{e}}
\newcommand{\verts}{\mathsf{v}}

\definecolor{purple}{rgb}{0.9,0,0.8}

\definecolor{gray}{rgb}{0.7,0.7,0.7}

\newtheorem{propo}{Proposition}[section]
\newtheorem{lemma}[propo]{Lemma}

\newtheorem{remark}[propo]{Remark}

\numberwithin{equation}{section}
\newtheorem{thm}[propo]{Theorem}
\newcommand{\bB}{\mathbb{B}}

\newcommand{\bE}{\mathbb{E}}

\newcommand{\bI}{\mathbb{I}}

\newcommand{\bP}{\mathbb{P}}

\newcommand{\bR}{\mathbb{R}}

\renewcommand{\AA}{\mathcal{A}}
\newcommand{\BB}{\mathcal{B}}
\newcommand{\CC}{\mathcal{C}}

\newcommand{\EE}{\mathcal{E}}
\newcommand{\FF}{\mathcal{F}}
\newcommand{\GG}{\mathcal{G}}

\newcommand{\KK}{\mathcal{K}}
\newcommand{\LL}{\mathcal{L}}

\newcommand{\NN}{\mathcal{N}}

\renewcommand{\SS}{\mathcal{S}}

\newcommand{\VV}{\mathcal{V}}
\newcommand{\WW}{\mathcal{W}}
\newcommand{\XX}{\mathcal{X}}

\newcommand{\ZZ}{\mathcal{Z}}

\newcommand{\sfC}{{\mathsf C}}
\newcommand{\sfE}{{\mathsf E}}
\newcommand{\sfK}{{\mathsf K}}
\newcommand{\sfP}{{\mathsf P}}
\newcommand{\sfS}{{\mathsf S}}

\newcommand{\sk}{\mathsf{k}}

\newcommand{\uu}{\underline}
\newcommand{\wt}{\widetilde}

\newcommand{\abbr}[1]{{\sc\lowercase{#1}}}
\newcommand{\eqnsection}{\renewcommand{\theequation}{\thesection.\arabic{equation}}
	\makeatletter \csname @addtoreset\endcsname{equation}{section}\makeatother}
\def\R{\mathbb R}

\def\N{\mathbb N}

\def\ind{{\mathbb I}}
\def\vep{\varepsilon}

\newcommand\blfootnote[1]{%
	\begingroup
	\renewcommand\thefootnote{}\footnote{#1}%
	\addtocounter{footnote}{-1}%
	\endgroup
}

\title[upper tail for constrained homomorphism counts]{Upper tail for homomorphism counts\\ in constrained sparse random graphs}

\author{Sohom Bhattacharya}
\address{Sohom Bhattacharya\newline Department of Statistics, Stanford University, California, USA,
\newline \tt{sohomb@stanford.edu}}

\author{Amir Dembo} 
\address{Amir Dembo\newline Department of Mathematics, Stanford University, California, USA, \newline \tt{adembo@stanford.edu}} 

\date{\today}

\begin{document}

	%
	%\title{Upper tail large deviation principle for constrained sparse random graphs}
	%\author{Amir Dembo \, Sohom Bhattacharya}
	%\institute{}
	%
	% typeset the header of the contribution
	%
	\begin{abstract}
		Consider the upper tail probability that the homomorphism count of a fixed graph $H$ 
		within a large sparse random graph $G_n$ exceeds its expected value by a fixed factor $1+\delta$. 
		Going beyond the Erd\H{o}s-R\'enyi model, we establish here explicit, sharp upper tail decay rates 
		for sparse random $d_n$-regular graphs (provided $H$ has a regular $2$-core), 
		and for sparse uniform random graphs.
		We further deal with joint upper tail probabilities for 
		homomorphism counts of multiple graphs $H_1,\ldots, H_\sk$ (extending the known results for $\sk=1$),
		and for inhomogeneous
		graph ensembles (such as the stochastic block model), we bound the
		upper tail probability by a variational problem analogous to the one that determines its decay
		rate in the case of sparse Erd\H{o}s-R\'enyi graphs.
	\end{abstract}
	
	\maketitle 
	\section{Introduction}

	\blfootnote{\textup{2010} \textit{Mathematics Subject Classification}: \textup{05C80}, \textup{60C05}, \textup{60F10}} 
	\blfootnote{\textit{Keywords and phrases}: \textup{graph homomorphism}, 
		%\textup{Variational problem}, 
		\textup{sparse random graph},
		\textup{d-regular graphs},
		\textup{large deviations}
	}
	\noindent 
	Let $\mathrm{Hom}(H,G)$ denote the number of copies of a \emph{connected} 
	graph $H=(V=[\verts(H)],E)$ present within 
	some other graph $G$ of $n$ vertices, which in terms of the adjacency matrix $A_G$ of $G$, is 
	% expressed by:
	\begin{equation} \label{eq:homdef}
	\mathrm{Hom}(H,G) \coloneqq \sum_{\substack{{\phi: [\verts(H)] \rightarrow [n]}}} \prod_{(k,l) \in E} A_G(\phi(k),\phi(l)).
	\end{equation}
	The upper tail homomorphism problem 
	for a given non-random, connected $H$, $\delta>0$ fixed and $G=G_n$ drawn from an ensemble 
	of random graphs on $n$ vertices with law $\bP$, is to estimate the tail probability rate 
	\begin{equation}\label{dfn:UT-pbm}
	\textrm{UT}(H,n,\delta) \coloneqq - \log \bP(\mathrm{Hom}(H,G_n) \geq (1+\delta) \mathbb{E}(\mathrm{Hom}(H,G_n))).
	\end{equation}
	This question has been extensively studied  in the context of Erd\H{o}s-R\'enyi (\abbr{ER}) binomial
	random graphs $\GG(n,p)$ (namely, when each edge independently selected 
	with probability $p$). First,  the 
	growth rate of $\textrm{UT}(H,n,\delta)$ as $n \to \infty$, 
	was established after considerable effort
	in the sparse regime, provided $p=p(n) \to 0$ at a slow
	enough rate
	(cf. \cite{chalog,doi:10.1002/rsa.20440,doi:10.1002/rsa.20382,jor,jr1,jr2,kimvu,vu} and related
	questions in the texts \cite{bol, jbook}).
	Then, relying on regularity and compactness 
	properties of the cut-metric, Chatterjee and Varadhan \cite{cv2,cv} proved the large deviation 
	principle (\abbr{LDP}) in the dense regime, where $p \in (0,1)$ is fixed. In particular, they 
	estimate $\textrm{UT}(H,n,\delta)$ by a variational problem over the space of graphons,
	within $1+o(1)$ relative error as $n \rightarrow \infty$. Following \cite{cv2,cv}, a 
	region where a constant graphon is optimal for such variational problems is characterized in  
	\cite{lz1} (and recently \cite{dhsen} establishes the \abbr{LDP}
	for uniform graphs of prescribed degrees $\{d_i\}$, in the dense regime where $d_i = O(n)$).
	
	To describe what is known in the sparse regime $p \to 0$,  for $X= (x_{ij})$
	from the collection $\XX_n$ of symmetric, $n \times n$ matrices  with $[0,1]$-valued 
	entries and zero main diagonal (ie $x_{ii} \equiv 0$), we let
	\[
	\mathrm{hom}(H,X) \coloneqq n^{-\verts(H)}p^{-\edges(H)} \sum_{\substack{{\phi: [\verts(H)] \rightarrow [n]}}} \prod_{(k,l) \in E(H)} X(\phi(k),\phi(l)) \,,
	\]
    denote the normalized weighted count of copies of a given connected $H$ having $\verts(H)=|V(H)|$ 
	vertices and  $\edges(H) = |E(H)|$ edges, further using
	\[
	I_p(X) \coloneqq \sum_{\substack{1 \leq i \neq j \leq n}} I_p(x_{ij}), \quad \textrm{ where } \quad 
	I_p(x) \coloneqq x \log \frac{x}{p}+(1-x)\log \frac{1-x}{1-p} \,,
	\]
	for the relative entropy of such $X \in \XX_n$ \abbr{wrt} the parameter $p$.
	The sparse regime poses extra difficulties, as graphon theory is no longer applicable.
	In lieu of that, a general scheme is introduced in \cite{chd} for approximating the partition function 
	of a Gibbs measure on the hypercube, whose potential 
	% (Hamiltonian)
	has a low complexity gradient. Utilizing this approach, \cite{chd} show that 
	for $\GG(n,p)$ under certain modest polynomial decay of $p(n)$, the upper tail rate 
	$\textrm{UT}(H,n,t-1)$ is for $t>1$ within $1+o(1)$ relative error of 
	\begin{align}\label{eq:0}
	\Phi_{n,p}(H,t) & \coloneqq \frac{1}{2}\inf\{I_p(X) : \quad X\in \XX_n, \quad \textrm{ hom}(H,X)\geq t\} 
	\end{align}
	(c.f. \cite{Chabook}). Invoking stochastic analysis tools, Eldan \cite{Eldan2018} obtains general conditions 
	for approximating such Gibbs measures by a mixture
	of products, and as a result relaxes somewhat the restriction of \cite{chd} on the decay of $p(n)$. 
	Beyond functions on the hyper-cube, \cite{jun} adapts the approach of \cite{chd} to general Banach
	spaces, whereas  Austin \cite{aus} utilizes information inequalities to extend Eldan's 
	results to arbitrary product spaces. Taking different, more direct approaches, 
	\cite{cod} and Augeri \cite{augeri}, independently establish large deviation results 
	for a host of spectral and geometric functionals on the hypercube, 
	and in particular 
	extend the $1+o(1)$ relative error between \eqref{dfn:UT-pbm} and \eqref{eq:0}, 
	to a much larger sparsity regime. For a specific sub-class of graphs $H$, even 
	smaller $p(n)$ is allowed in \cite{samotij} which develops for this  
	a method of entropic stability, and in its followup \cite{bb} which extends 
	these results to a larger sub-class of graphs $H$.
	
	While the analysis in all these works relies 
	on the independence and homogeneity 
	inherent of $\GG(n,p)$, as well as the simpler geometry of the tail event when only one $H$ is 
	considered, we dispense here from most of these restrictions. Specifically, 
	\Cref{th:correg+unif}
	expands our understanding of the upper tail problem, by considering 
	% constrained ensembles, specifically
	random graphs $G_n^{(m)}$ chosen uniformly among all graphs of $n$ vertices and $m$ edges,
	as well as uniformly chosen random regular graphs $G_n^d$ on $n$ vertices, each having the same 
	degree $d=d_n$. Similarly to the \abbr{ER}-model, for both $G_n^d$ and $G_n^{(m)}$ 
	the relevant large deviation events correspond to planting specific small structures
	within $G_n$. However, the degree constraints sometimes 
	prohibit the planting strategy optimal for the \abbr{ER}-case, requiring us to achieve the 
	excess count by planting multiple disjoint small structures and  
	% have new ideas and 
	to develop in the proof of \Cref{thm:regvar}
	new tools for the study of the relevant variational problem. 
	Turning back to edge independence, \Cref{th:sbmprob} shows that also 
	for an inhomogeneous setting (such as the stochastic block model), the upper tail 
	probability decay rate boils down to a suitable variational problem, while within
	the \abbr{ER}-model $\GG(n,p)$, as well as for the uniformly random graphs $G_n^{(m)}$,
	\Cref{prop:jut} provides the complete solution of the upper tail problem for 
	joint counts of graphs $\{H_i, i=1,\ldots,k \}$.
	% (even of varying maximal degrees). 
	
	Throughout we adjust across different ensembles for equivalent sparsity. 
	That is, we parameterize $G_n^d$ via $p=d/n$ and likewise parameterize
	$G_n^{(m)}$ via $p=m/{n \choose 2}$. Denoting by $\Delta=\Delta(H)$ the maximal degree in $H$,
	and writing hereafter $a_n \sim b_n$ whenever $a_n/b_n = 1 + o(1)$, 
	recall that for any $p=p(n) \gg n^{-1/\Delta}$ one has that
	$\bE[\mathrm{Hom}(H,G_n)] \sim n^{\verts(H)} p^{\edges(H)}$ in the \abbr{ER}-model.
	It is easy to see that this applies also for  $G_n^{(m)}$ in such regime of $p(n)$,
	while \cite[Corollary 2.2]{ksv} establishes
	the same conclusion for $G_n^d$. 
	Thus, using the normalized $\mathrm{hom}(H,G)$ for a random graph $G$ on $n$
	vertices from either of our ensembles, and setting $\bP^{(m)}$ and
	$\bP^{d}$ for the laws of $G_n^{(m)}$ and $G_n^d$, we have
	in analogy with \eqref{dfn:UT-pbm}, the upper tails
	\begin{equation}\label{dfn:UT-pbm-cons}
	\begin{aligned}
	\mathrm{UT}^{(m)}(H,n,\delta) &:= - \log \bP^{(m)} (\mathrm{hom}(H,G_n^{(m)}) \geq 1+\delta) \,, \\ 
	\textrm{UT}^d(H,n,\delta) &:= - \log \bP^{d}(\mathrm{hom}(H,G^d_n) \geq 1+\delta) \,.
	\end{aligned}
	\end{equation}
	Recall the collection $\XX_n$ of adjacency matrices for
	% the set $\GG_n$ of
	$[0,1]$-weighted simple graphs on $n$ vertices, while 
	%  the subsets 
	\[
	\XX^{(m)}_n := \{(x_{ij}) \in \XX_n : \sum_{i,j=1}^n x_{ij}=m \} \,,
	\quad
	\XX^d_n := \{ (x_{ij}) \in \XX_n : \sum_{i=1}^n x_{ij} =d, \quad 1 \leq j \leq n \}\,,
	\]
	indicate such matrices for graphs of a given total weight, or of a given constant vertex
	weight, respectively. Indeed, we show in the sequel that the corresponding rate functions for homomorphism counts
	within uniform random graphs and within random $d$-regular graphs are:
	\begin{align}\label{eq:1}
	\Phi^{(m)}_n(H,t) 
	& \coloneqq \frac{1}{2}\inf\{I_p(X) : \quad X\in \XX^{(m)}_n, \quad \mathrm{ hom}(H,X)\geq t\}, \\
	% 2m=n^2p,
	%\begin{align}
	\label{eq:2}
	\Phi^d_n(H,t) & \coloneqq \frac{1}{2}\inf\{I_p(X) : \quad X\in \XX^{d}_n, \qquad 
	\mathrm{ hom}(H,X)\geq t\}
	% d = np.
	\end{align}
	(bounds on $UT^{(m)}(H,n,\delta)$ are given in \cite[Thm. 4.1]{jor}, 
	with $\Phi^{(m)}_n(H,t)$ appearing in \cite[Prop. 3.3]{dl}, 
	where the asymptotic of $\mathrm{UT}^{(m)}(H,n,\delta)$ is established for 
	the very slow decay $p(n) \gg (\log n)^{-1/(2\edges(H))}$).
	% Denoting by $\Delta$ the maximal degree in $H$, 
	
	Recall that, the $k$-core of a graph is the maximal subgraph  within which 
	every vertex has degree at least $k$. It is not hard to check that  
	the rate of growth of
	each of 
	the variational problems \eqref{eq:0} and \eqref{eq:1} is $a_{n,p}\coloneqq n^2p^\Delta \log(1/p)$
     (which is also the rate for \eqref{eq:2} when the $2$-core of $H$ is $\Delta$-regular).
	More precisely,  it is shown in \cite{bglz} 
	that for any $\delta>0$, connected graph $H$ of maximal degree 
	$\Delta \geq 2$ and $n^{-1/\Delta} \ll p =o(1)$, for the normalized variational problem 
	% \eqref{eq:0} via
	$\phi_{n,p} (\cdot) \coloneqq a_{n,p}^{-1} \Phi_{n,p} (\cdot)$ one has that 
	\begin{equation}\label{eq:ervar}
	\lim\limits_{n \rightarrow \infty} \phi_{n,p}(H,1+\delta)= c(H,\delta) \coloneqq \begin{cases}
	\min \{\theta, \frac{1}{2} \delta^{2/\verts(H)} \}, & \text{for regular } H, \\
	\theta, & \text{ otherwise}
	\end{cases}
	\end{equation}
	(with triangle counts, namely $H=\sfC_3$, settled earlier in \cite{lz2}).
	Here $\theta = \theta(H, \delta)$ is the unique positive solution of $\sfP_{H^\star}(\theta)=1+\delta$, 
	for the independence polynomial $\sfP_{H^\star}(\cdot)$ of the sub-graph 
	$H^\star = H[V^\star]$ induced by $H$ on its set of vertices $V^\star \subset V$ of 
	% maximal 
	degree $\Delta$. The two expressions on the \abbr{rhs} of \eqref{eq:ervar} correspond to 
	planting a relatively small clique, at rate $\delta^{2/\verts(H)}$, 
	or hub (=anti-clique), at rate $\theta$. This interpretation is further  
	detailed in \Cref{rmk:clique+hub}, where for joint $\sk \ge 2$ homorphism counts
	one often gets a 
	% more general 
	clique+hub planting as the optimal solution. We further note in passing 
	that such a variational problem for lower tails is addressed in \cite{lower}, with \cite{ap}
	and \cite{bg} studying analogous variation problems for arithmetic progressions on random sets 
	and for the upper tail of edge eigenvalues in case of the \abbr{er}-model.
	
	Utilizing the same normalization, 
	%to define $\phi_n^{(m)}$ and $\phi_n^d$,
	we turn  to  the explicit solution of 
	%the variational problems 
	\eqref{eq:2}, noting first that for any $X \in \XX_n^d$ the value of 
	$\mathrm{hom}(H,X)$ is invariant to removal from $H$ any vertex of
	 degree one, since summing
	 	over the index of such a vertex merely gives a multiplicative factor 
	 	$d=np$. In particular, if $H$ is a tree then 
	$\textrm{UT}^d(H,n,\delta) =  - \Phi^d_n(H,1+\delta)=-\infty$ for any
	$\delta>0$, while replacing any other $H$ by its 
	non-empty $2$-core changes neither $\textrm{UT}^d(H,n,\delta)$ 
	nor $\Phi^d_n(H,t)$. Thus, whenever we consider $G_n^d$, we assume \abbr{wlog} 
	that the minimal degree of $H$ is at least two.
	\begin{thm}\label{thm:regvar}
		For $\delta > 0$, a connected graph $H$ having 
		no vertices of degree one and maximal degree $\Delta \geq 2$,
		and for any $n^{-\frac{1}{\Delta}} \ll p := \frac{d}{n} = o(1)$, 
		\begin{equation}\label{eq:regclq}
		\lim\limits_{n \rightarrow \infty} \phi^d_n(H,1+\delta)= c^d(H,\delta) \coloneqq 
		\frac{1}{2} \begin{cases}
		\lfloor \delta \rfloor + \{\delta\}^{2/\verts(H)}, &  \qquad \Delta =2, \\
		\qquad \quad \delta^{2/\verts(H)}, &  \Delta\textrm{-}\mathrm{regular} \;H, \; \Delta \geq 3, \\
		\qquad \quad \infty, & \qquad \mathrm{otherwise,}
		\end{cases}
		\end{equation}
		where $\phi_n^d(\cdot) := a_{n,p}^{-1} \Phi_n^d(\cdot)$ is the normalized value of \eqref{eq:2}
		and $\{\delta\}$ denotes the fractional part of $\delta$.
	\end{thm}
	
\begin{remark}\label{rmk: nochange} 
		The degree constraints of $G_n^d$ 
		rule out any hub and further limit the allowed planted clique size. Hence
		our result in \Cref{thm:regvar}, corresponding for regular $H$ to the planting of
		$\lceil \delta \rceil$ disjoint cliques when $\Delta=2$ (higher values of
		$\Delta$ require smaller clique size, so our
		size limit no longer affects the solution, see $X_n^\star$ of \eqref{eq:opt-cyc} versus
		\eqref{eq:opt-clique}). When $H$ as in \Cref{thm:regvar} is irregular with $\Delta \ge 3$,
		as shown already in \eqref{eq:regclq} even
	    the growth rate of $\Phi_n^d(H,1+\delta)$ differs from that for  the \abbr{ER}-model. 
		While our paper was under review, 
		\cite[Theorem 2.7]{ben} established the growth rate of $\Phi_n^d(H,1+\delta)$
		for a large class of irregular graphs $H$.
	
	\end{remark}

	As promised before, we next show that the variational problems we solved in
	\Cref{thm:regvar} control the asymptotic rates of 
	$\text{UT}^d(H,n,\delta)$, whereas $\text{UT}^{(m)}(H,n,\delta)$ follow the same 
	asymptotic as $\text{UT}(H,n,\delta)$ (indeed, the relatively small structures 
	which dominate the upper tail variational problems for small $p(n)$, are
unaffected by a global edge constraint). 		
			
	\begin{propo}\label{thm:regprob}
		For a graph $H=(V,E)$ of maximal degree $\Delta \geq 2$, set 
		$$
		\Delta_{\star} (H) \coloneqq \frac{1}{2} \max_{\{v_1,v_2\} \in E(H)} \{ \deg_H(v_1)+\deg_H(v_2)\} \geq 1\,.
		$$
		Then, denoting by $\sfC_l$ a cycle of length $l \ge 3$, for fixed $t>1$ and any 
		\begin{equation}\label{eq:p-range}
		1 \gg p \gg 
		\begin{cases} & \max(n^{\frac{2}{l}-1},\frac{(\log n)^{\frac{l}{2l-4}}}{\sqrt{n}}) \,, \qquad \qquad H=\sfC_l, \\
		&  n^{-1/(2 \Delta_{\star}(H))}(\log n)^{1/(2\Delta_\star(H))} \,, \;\;\;  \mbox{ otherwise},
		\end{cases}
		\end{equation}
		one has that for some $\kappa_n \uparrow \infty$,
		\begin{align}\label{eq:3}
		a_{n,p}^{-1} \log\bP^{(m)}(\mathrm{hom}(H,G^{(m)}_n) \geq t) &\leq -\phi_{n,p}(H,t-o(1))+o(1), \\
		\label{eq:4}
		a_{n,p}^{-1} \log\bP^{d}(\mathrm{hom}(H,G^d_n) \geq t) &\leq -(\phi^d_n(H,t-o(1))-o(1) )\wedge \kappa_n.
		\end{align}
	\end{propo}
\begin{remark}\label{rem:samot}
Similarly to \Cref{thm:regvar}, we replace $H$ in \eqref{eq:4} by its
 $2$-core before setting $\Delta$, $a_{n,p}$ and the  allowed range \eqref{eq:p-range} for $p(n)$.
Since we get \eqref{eq:3} by a direct comparison with the \abbr{ER}-model, 
%$\GG(n,p)$, 
for a $\Delta$-regular $H$ one has that \eqref{eq:3} holds
upto $p \gg n^{-1/\Delta}$, by relying on \cite[Theorem 1.5]{samotij} or \cite[Theorem 1.2]{bb} instead of \cite{cod}. 
A similar improvement may likewise hold  in \eqref{eq:4} 
when the $2$-core of $H$ is $\Delta$-regular.
\end{remark}

Building on \Cref{thm:regvar} and \Cref{thm:regprob}, we get
	the following.
	\begin{thm}$~$\label{th:correg+unif}
 Fix $\delta>0$ and connected graph $H$.
		\newline
		(a). Replacing $H$ by its $2$-core of maximal degree $\Delta \ge 2$,
		for any $p=\frac{d}{n}$ as in \eqref{eq:p-range},
		\begin{equation}\label{eq:8}
		\lim\limits_{n\rightarrow\infty} a_{n,p}^{-1} \, \mathrm{UT}^d(H,n,\delta) = c^d(H,\delta).
		\end{equation}
		(b). Assuming $\Delta(H)=\Delta \ge 2$, for any $p=m/{n \choose 2}$ as in \eqref{eq:p-range},
		\begin{equation}\label{eq:7}
		\lim\limits_{n\rightarrow\infty} a_{n,p}^{-1} \, \mathrm{UT}^{(m)}(H,n,\delta) = c(H,\delta).
		\end{equation}
	\end{thm}
	\noindent The stated lower bounds on the limits in \eqref{eq:8}-\eqref{eq:7} 
	are  immediate from \Cref{thm:regvar} and \Cref{thm:regprob}. We attain 
	the complementary upper bounds by planting cliques or a hub 
	%(=anti-clique),
	according to the explicit optimal strategies $X_n^\star$ we use in \Cref{opt-reg}
%	\Cref{thm:regvar} 
	or those used in proving \eqref{eq:ervar}, as a 
	by product of which we further deduce that for any 
	% $\delta >0$, connected $H$ of maximal degree $\Delta \geq 2$, and
	$n^{-\frac{1}{\Delta}} \ll p := m/{n \choose 2} =o(1)$, 
	\begin{equation}\label{was-th:unifvar}
	\lim\limits_{n \rightarrow \infty} a_{n,p}^{-1} \Phi_n^{(m)}(H,1+\delta) = c(H,\delta) \,,
	\end{equation}
	for $c(H,\delta)$ given on the \abbr{rhs} of \eqref{eq:ervar}.
	\begin{remark} Having only the limiting upper tail rate, as in \Cref{th:correg+unif}, 
		is not enough for precise information about the law of the (rare) graphs $G_n$ for which 
		$\mathrm{hom}(H,G_n)$ exceeds its mean by factor $1+\delta$. Nevertheless, our results provide 
		additional evidence that such graphs be typically close to a sample from the original 
		%random
		ensemble with an added structure of suitable $o(n)$-size that mimic the explicit 
		optimizers we use in the proof of that theorem.
	\end{remark}

	Next, consider the inhomogeneous \abbr{ER} setting, where given  
	probability vectors $\{\bm{\alpha}^{(n)}\}$ of length $\ell$ each,
	the vertices of $G_n=G_n^{[\ell]}$ are split to $\ell$ blocks, having sizes $\alpha^{(n)}_r n$ for $1 \le r \le \ell$,
	and the edges between vertices within the $r$-th and $r'$-th blocks are formed 
	independently, with probability $c^{(n)}_{r r'} \, p$. Assuming that $\{ c_{r r'}^{(n)} \}$ are 
	uniformly bounded, $\{ \alpha^{(n)}_1 \}$ and $\{ c_{1 1}^{(n)} \}$ are bounded away from zero, 
     while $p=p(n)\rightarrow 0$ at a suitable rate, we  
	denote by $\bP^{[\ell]}$ the law of the resulting random graph $G_n^{[\ell]}$, 
	parameterized by the symmetric $n \times n$ matrix $\bm{p}=(p_{ij})$ of entry values
	$\{ c^{(n)}_{r r'} \, p(n) \}$ as above, and 
	for $X \in \XX_n$, set
	\[
	I_{\bm{p}} (X) :=\sum_{1 \le i \ne j \le n} I_{p_{ij}} (X_{ij}) \,.
	\]
	Analogously to \Cref{thm:regprob}, we next show 
	that the upper tail event for $G_n^{[\ell]}$ is characterized by 
	% the rate function 
	\begin{equation}\label{dfn:Phi-s}
	\Phi_{n,\bm{p}}^{[\ell]}(H,t) := \frac{1}{2}\inf \{I_{\bm{p}}(X) :  \quad
	X \in \XX_n, \quad  \mathrm{hom}(H,X)\geq t \, b_H
	\}.
	\end{equation}
    The constant  $b_H = b_H^{(n)} \sim \mathbb{E}^{[\ell]} [\, \mathrm{hom}(H,G_n) \,]$
	denotes the following sum over partitions $\{S_r\}$ of $V(H)$ 
	to $\ell$ parts (possibly empty), 
	\begin{equation}\label{dfn:kappa-H}
	b_H \coloneqq \sum_{\{S_r\}} \;\; \prod_{r=1}^\ell \alpha_r^{|S_r|} 
	\prod_{1 \le r \leq r' \le \ell} c_{r r'}^{\edges(H[S_r \rightarrow S_{r'}])} \,,
	\end{equation}
     where $\edges(H[S_r \rightarrow S_{r'}])$ count edges between blocks $S_r$ and $S_{r'}$ 
     % of $V(H)$ 
     (and $\{\alpha^{(n)}_r,c^{(n)}_{r,r'}\}$ may depend on $n$).
	\begin{propo}\label{th:sbmprob}
		For fixed $t>1$, a graph $H$ of maximal degree $\Delta \geq 2$
		and $p=p(n)$ of \eqref{eq:p-range}, 
		% one has that
		\begin{equation}\label{eq:sbm1}
		a_{n,p}^{-1}\log\bP^{[\ell]}(\mathrm{hom}(H,G^{[\ell]}_n) \geq t \, b_H ) \leq -\phi_{n,\bm{p}}^{[\ell]}(H, t-o(1))+o(1).
		\end{equation}
	For $1 \gg p \gg n^{-1/(2\Delta_\star(H)} (\log n)^{1/(2\Delta_\star(H)}$ we have in addition that 
	 	\begin{equation}\label{eq:sbm2}
		a_{n,p}^{-1}\log\bP^{[\ell]}(\mathrm{hom}(H,G^{[\ell]}_n) \geq t \, b_H ) \geq -\phi_{n,\bm{p}}^{[\ell]}
		(H, t+o(1))-o(1).
		\end{equation}
	\end{propo}
	\begin{remark} Additional work should yield that \eqref{eq:sbm2} holds 
	for $p(n)$ in the full range of \eqref{eq:p-range}.
	\end{remark}

	Turning to the joint upper tail for 
	% $k$-dimensional 
	the vector 
	$\mathrm{hom}(\underline{H},G) \coloneqq (\mathrm{hom}(H_1,G),\ldots, (\mathrm{hom}(H_\sk,G))$
	corresponding to a given collection $\underline{H} \coloneqq (H_1, \ldots, H_\sk)$ of connected graphs
	$\{H_i\}$, we endow $\R^\sk$ with the usual coordinate-wise partial orders $\ge$ and $>$. 
	As we show next, for \abbr{ER}-model $\GG(n,p)$, whose law we denote hereafter by 
	$\bP_p$, the rate function at 
	$\uu{t} = (t_1,\ldots,t_\sk) > {\bf 1} := (1,\ldots,1) \in \R_+^\sk$ is then 
	\begin{equation}\label{eq:multvar}
	\Phi^{\sk}_{n,p}(\underline{H}, \underline{t}) := \frac{1}{2}\inf \{I_p(X) :  \quad 
	X\in \XX_n, \quad \mathrm{hom}(\underline{H},X) \geq \underline{t} \}
	\end{equation}
	(compare with \eqref{eq:0} which corresponds to $\sk=1$).
	\begin{propo}\label{th:multprob}
		For $\sk' \ge 1$ let $\Delta := \min_{j \in [\sk']} \{ \Delta(H_j) \} \ge 2$ denote the minimal value 
		among the maximal degrees of given connected graphs $\{H_j, j \in [\sk']\}$. Assume 
		\abbr{wlog} that $\Delta(H_i) = \Delta$ iff $i \in [\sk]$ for some $\sk \in [\sk']$
		and set $\phi_{n,p}^{\sk} (\cdot) := a_{n,p}^{-1} \Phi_{n,p}^{\sk} (\cdot)$ for 
		such $\sk$ and the scaling $a_{n,p}$ induced by $\Delta$.
		Then, for any $\underline{t} \in [1,\infty)^{\sk'}$ and $p=p(n)$ 
		in the intersection of ranges \eqref{eq:p-range} applicable to $H_i$, $i \in [\sk]$,
		%  one has that 
		\begin{equation}\label{eq:multprob}
		a_{n,p}^{-1}\log\bP_p(\mathrm{hom}(\underline{H},G_n) \geq \underline{t}) \leq
		-\phi^\sk_{n,p}(\pi_k(\uu{H}),\pi_\sk(\uu{t})-o(1)) +o(1),
		\end{equation}
		where $\pi_\sk$ denotes the restriction to the first $\sk$ coordinates (both for $\uu{H}$ and
		on $[1,\infty)^{\sk'}$).
	\end{propo}
	\noindent We complement \Cref{th:multprob} by the following explicit solution of
	the variational problem \eqref{eq:multvar}.
	\begin{propo}\label{th:multvar}
		Fix $\sk \ge 1$, $s \ge 0$ and suppose the 
		connected graphs $\{H_i, i \in [\sk]\}$ have the \emph{same} maximal degrees
		$\Delta(H_i)=\Delta \geq 2$ and $H_i$ is $\Delta$-regular iff $i \le s$. Then,
		for any $\uu{\delta} \in \R_+^\sk$ and $n^{-1/\Delta} \ll p=o(1)$ 
		\begin{align}\label{eq:multans}
		\lim_{n \rightarrow \infty} \phi^\sk_{n,p}(\underline{H}, {\bf 1} + \underline{\delta}) 
		&= c(\uu{H},\uu{\delta}) \nonumber \\
		& \coloneqq \min_{x,y \geq 0} \{ x + \frac{1}{2} y^2 :  \;
		\sfP_{H^{\star}_{i}}(x)+ \ind_{\{i \le s\}} y^{\verts(H^{\star}_{i})} \geq 1+ \delta_{i}, \;  i \le \sk \}.
		\end{align}
	\end{propo}
	\begin{remark}\label{rmk:clique+hub}
		On the \abbr{rhs} of  \eqref{eq:multans} we have the normalized size of a planted 
		hub $(=x)$ and a planted clique $(=y)$, in the limiting $n \to \infty$ solution of \eqref{eq:multvar}.  
		As shown in \cite{bglz}, for $\sk=1$ such optimum is always attained for $x=0$ or $y=0$, 
		yielding \eqref{eq:ervar}. In contrast,  for $\sk \ge 2$ the optimum in general 
		has both $x>0$ and $y>0$, corresponding as mentioned before, 
		to simultaneously planting both a clique and a hub (for example, for $H_1=\sfK_3$ and $H_2=\sfK_{1,2}$ we 
		are to minimize $x+ \frac{1}{2} y^2$ in \eqref{eq:multans} subject to 
		$1+3x+y^3 \geq 1+\delta_1$ and $1+ x \geq 1+ \delta_2$. 
		The latter constraint rules out $x=0$, and for $\delta_1 - 3 \delta_2 > 27/8$, 
		taking $(x,y)=(\delta_2,(\delta_1-3 \delta_2)^{1/3})$ is better than 
		the hub solution $(x,y)=(\delta_1/3,0)$).
	\end{remark}
	
	\noindent
	Building on Propositions \ref{th:multprob} and \ref{th:multvar} we establish
	the following sharp joint upper tail asymptotic.
	\begin{thm}\label{prop:jut}
		With $c(\underline{H},\underline{\delta})$ given by  \eqref{eq:multans}, we have
		in the setting of \Cref{th:multprob},  that 
		\begin{equation}\label{eq:ut-joint-er}
		\lim_{n \rightarrow \infty} a_{n,p}^{-1} \log\bP_p(\mathrm{Hom}(\underline{H},G_n) \geq {\bf 1}+\underline{\delta}) 
		= - c(\pi_\sk(\underline{H}),\pi_\sk(\underline{\delta}))\,.
		\end{equation}
		Further, the same applies for the law $\bP^{(m)}$ of 
		the uniformly random  graph
		$G_n^{(m)}$.
	\end{thm}
	
	\begin{remark} We believe that	the analog of \eqref{eq:ut-joint-er}
	holds for $G_n^d$, provided $c(\uu{H},\uu{\delta})$ of 
		\eqref{eq:multans} is replaced by $\max_{i=1}^{\sk} \{ c^d(H_i,\delta_i) \}$. Indeed, 
		our proof of \Cref{thm:regvar} extends to show that if in addition 
		$\Delta(H_i)=\Delta \ge 2$ for all $i \le \sk$,  then for any $\uu{\delta}$, $p$ as
		in \Cref{th:multvar},
		\begin{equation}\label{eq:multcycreg}
		\lim\limits_{n \rightarrow \infty} a_{n,p}^{-1} \, \inf \{ \frac{1}{2} I_p(X) : \;
		X\in \XX^d_n, \;\; \mathrm{hom}(\underline{H},X) \geq {\bf 1} + \underline{\delta} \} 
		= \max_{i=1}^\sk \{ c^d(H_i,\delta_i) \} \,.
		\end{equation}
	\end{remark}
	
	The bulk of this paper is \Cref{sec:unif}, where we settle \Cref{thm:regvar},  \Cref{thm:regprob} and 
	\Cref{th:correg+unif} on the upper tail problem for random $d$-regular and uniformly random graphs.
    The short \Cref{sec:sbm} then establishes \Cref{th:sbmprob} about the inhomogeneous random graph 
    $G_n^{[\ell]}$, while \Cref{sec:joint} deals with joint homomorphism counts,
	% in the \abbr{ER}-model
	proving Propositions \ref{th:multprob}-\ref{th:multvar} and \Cref{prop:jut}.

	\section{Uniform random and random regular graphs}\label{sec:unif}
	\subsection{Proof of \Cref{thm:regvar}}
	We start by exhibiting in {\bf Step 1}, ie. in
	\Cref{opt-reg}, 
	an optimal strategy 
	for $\Delta$-regular $H$, thereby upper bounding $\lim_n \phi_n^d(H,1+\delta)$. 
	The technically challenging lower bounds 
	are then separately proved when $\Delta=2$ (in {\bf Step 2}), and when $\Delta \ge 3$ (in {\bf Step 3}). 
	\begin{propo}\label{opt-reg}
		% Let $C_l$ be the cycle of length $l \geq 3$. 
		Fixing $\delta > 0$ and connected $\Delta$-regular $H$, $\Delta \ge 2$, if
		$p=\frac{d}{n} \to 0$, $d \to \infty$, then
		\begin{equation}\label{eq:regcyc}
		{\Phi^d_n(H,1+\delta)} \leq c^d(H,\delta) a_{n,p} (1+o(1)).
		\end{equation}
	\end{propo}
	\begin{proof} Consider first $\Delta=2$, for which $H=\sfC_l$ must be a cycle of length $l \ge 3$.
		In this case, our candidate for \eqref{eq:2} is the block adjacency matrix 
		$X_n^\star \in \XX_n$, of the form
		% of $[0,1]$-weighted graph 
		
		\begin{equation}\label{eq:opt-cyc}
		X_n^\star := 
		\begin{bmatrix}
		\textbf{1} & \mathbf{0} & 
		%\mathbf{0} & 
		\mathbf{0} & \cdots & \mathbf{0} \\ 
		\textbf{0} & \mathbf{1}
		& \mathbf{0} 
	    & \cdots & \mathbf{0} \\
		% \textbf{0} & \textbf{0} & \mathbf{1} &\cdots & \cdots & \cdots \\
		\cdots & \cdots & \cdots 
		%&\cdots 
		& \cdots & \cdots \\
		\mathbf{0} & \mathbf{0} & \cdots 
		% &\cdots
		&\textbf{1} & \textbf{r}\\
		\mathbf{0} & \mathbf{0}
		% & \cdots
		&\cdots &\textbf{r} & \textbf{q}
		\end{bmatrix}\,,
		\end{equation}
		where we have $\lceil \delta \rceil$ principal blocks of ones, denoted by \textbf{1},  
		the first $\lfloor \delta \rfloor$ of which are of the maximal size $d+1$ each, 
		while the last block is of a size $s_1$ such that $s_1 \sim \{\delta\}^{1/l} d$. Here $\textbf{q}$ denotes a block matrix with zeros on the diagonal and $q$ on the off-diagonals, and $\textbf{r}$ denotes the matrix with
			all elements $r$.
		Setting $s:=\lfloor \delta \rfloor (d+1) + s_1$,
		the row-sum constraint of $\XX_n^d$ is satisfied by $X_n^\star$, provided 
		$r$ and $q$ are such that
		\begin{align*}
		d= s_1 - 1 + (n- s) r =  r s_1  + q (n- s-1) \,.
		\end{align*}
		Since $d=np$, this results with 
		\begin{equation}\label{eq:r-q-val}
		r=\frac{np-s_1+1}{n- s}\,,  \qquad \qquad q=\frac{np- s_1 r}{n - s-1} \,.
		\end{equation} 
		As $p=p(n) \to 0$, it follows from \eqref{eq:r-q-val} that eventually $r \le p$ and 
		furthermore $q/p \rightarrow 1$. 
		We denote the homomorphism density of $H=([\verts(H)],E)$ in $X \in \XX_n$, by
		\begin{equation}\label{eq:thg}
		t(H,X) \coloneqq n^{-\verts(H)} \textrm{ Hom}(H,X) = n^{-\verts(H)} \sum_{\substack{1 \leq i_1, \cdots i_\verts(H) \leq n}}
		\prod_{(k,l) \in E} x_{i_k,i_l}.
		\end{equation}
		Now, with $\edges(\sfC_l)=\verts(\sfC_l)$, upon   
		considering only contributions when all vertices of $\sfC_l$
		are in the same principal block of $X^\star_n$,  we find that
		\begin{align}\label{eq:tH-cyc}
		t(\sfC_l,X^\star_n) & \geq \lfloor \delta \rfloor \Big( \frac{d}{n} \Big)^{\verts(\sfC_l)} + \Big(\frac{s_1-1}{n}\Big)^{\verts(\sfC_l)} + 
		\Big( \frac{n-s-1}{n}\Big)^{\verts(\sfC_l)} q^{\edges(\sfC_l)} \nonumber \\
		& \sim (\lfloor \delta \rfloor + \{\delta\})p^{\verts(\sfC_l)}+p^{\edges(\sfC_l)} = (1+\delta)p^{\edges(\sfC_l)},
		\end{align}
		as required in \eqref{eq:2}.
		As for the entropy of $X_n^\star$, clearly
		\begin{equation}\label{eq:13}
		I_p(X_n^\star) \le (\lfloor \delta \rfloor + \{\delta\}^{2/l}) n^2 p^2 I_p(1) + \lfloor \delta \rfloor
		n^2 p I_p(0)  + 2 n s_1 I_p(r) +n^2 I_p(q) \,.
		\end{equation}
		With $I_p(1) = \log (1/p)$, the first term on the \abbr{rhs} is precisely $2 c^d(H,\delta) a_{n,p}$ 
		(see the \abbr{rhs} of \eqref{eq:regclq} for $\Delta=2$). Since $I_p(0)= 
		% (1-p) \log (1/(1-p)) = 
		o(p \log(1/p))$, the second term on the \abbr{rhs} is $o(a_{n,p})$. Recall that eventually $r \le p$,
		hence $I_p(r) \le I_p(0)$
		%because $I_p(.)$ decreases in $[0,p]$ and increases in $[p,1]$) and $s_1 \sim \{\delta\}^{1/l}np$, 
		and with $s_1 \le np$, the third term is similarly $o(a_{n,p})$. As for the last term, note that 
		$I'_p(x)= \log \Big\{ \frac{x(1-p)}{p(1-x)} \Big\}$ is uniformly bounded over 
		$x \in [p/2,2p]$ and $p \le 1/3$. Further, from  \eqref{eq:r-q-val} we have that 
		\begin{equation}\label{eq:q-p-err}
		q-p= \frac{p (s-1)  - r s_1}{n-s-1}= O(p^\Delta)\,.
		\end{equation}
		In particular, as $p \to 0$, eventually $q \in [p/2,2p]$. With $I_p(p)=0$, we then have that 
		\begin{equation}\label{eq:qpbound}
		I_p(q) \leq |q-p| \sup_{x\in [\frac{1}{2}p,2p]} |I'_p(x)| = O(p^\Delta) = o(p^\Delta \log(1/p))     
		\end{equation}
		so the last term on the \abbr{rhs} of \eqref{eq:13} is also $o(a_{n,p})$.
		
		\medskip
		In case $\Delta \geq 3$ it suffices to plant a single clique. Specifically, consider 
		$X^\star_n$ as in \eqref{eq:opt-cyc}, except for having now \emph{only its single, last
			block of ones}, namely, set an integer $s=s_1 \sim \delta^{1/\verts(H)} n p^{\Delta/2}$ and
		\begin{equation}\label{eq:opt-clique}
		X_n^\star = 
		\begin{bmatrix}
		\textbf{1} & \mathbf{r} \\ 
		\textbf{r} & \mathbf{q} 
		\end{bmatrix} \,.
		\end{equation}
		Here $\Delta/2>1$ and $p \to 0$, so $s=s_1 = o(d)$ 
		regardless of the fixed value of $\delta$, allowing us to
		set $r,q \in [0,1]$ per \eqref{eq:r-q-val}, provided $p$ is small enough,
		to guarantee that $X_n^\star \in \XX^d_n$. 
		In contrast, for $\Delta=2$ this would have given $s=O(d)$, hence the need for 
		the more elaborate construction \eqref{eq:opt-cyc}.
		 Next, here $\edges(H)=\Delta \verts(H)/2$ and   
		considering contributions when all vertices of $H$
		are in the same principal block of $X^\star_n$,  we find similarly to \eqref{eq:tH-cyc}, that 
		\begin{align}\label{eq:exess-clique}
		t(H,X^\star_n)  \geq & \Big(\frac{s_1-1}{n}\Big)^{\verts(H)} + 
		\Big( \frac{n-s-1}{n}\Big)^{\verts(H)} q^{\edges(H)}  \nonumber \\
		& \sim
		\delta p^{\Delta \verts(H)/2}+p^{\edges(H)} = (1+\delta)p^{\edges(H)},
		\end{align}
		as required in \eqref{eq:2}. Further, similarly to \eqref{eq:13}, we now have  
		\[
		I_p(X_n^\star) \le s_1^2 I_p(1) +  2 n s_1 I_p(r) +n^2 I_p(q) \,,
		\]
		where the first term on the \abbr{rhs} is $2 c^d(H,\delta) a_{n,p} (1+o(1))$. 
		It is easy to see that here
		\begin{equation}\label{eq:r-p-err}
		p-r =  \frac{(1-p)s_1-1}{n-s_1} = O(p^{\Delta/2}), \qquad   
		q-p = \frac{(p-r) s_1-1}{n-s_1-1} = O(p^{\Delta}) \,.
		\end{equation}
		With $q/p \to 1$ satisfying \eqref{eq:q-p-err}, as argued in case
		$\Delta=2$, here again $n^2 I_p(q) = O(n^2 p^\Delta) = o(a_{n,p})$. Now also
		$r/p \to 1$, so by the same reasoning $I_p(r) = O(p^{\Delta/2})$. With 
		the corresponding term in our bound on $I_p(X_n^\star)$ being $o(a_{n,p})$,
		this completes the proof of the proposition.
	\end{proof}
	
	\noindent
	{\bf Step 2} (Lower bound $\Delta =2$):  If $\Delta=2$ then $H=\sfC_l$ for some 
	$l \ge 3$, and our starting point in bounding below
	\begin{align*}
	\Phi^d_n(\sfC_l,1+\delta)= \frac{1}{2}\inf\{I_p(X), X\in \XX^d_n, \textrm{ hom}(\sfC_l,X)\geq 1+\delta\},
	\end{align*}
	is the inequality 
	\begin{equation}\label{eq:50}
	\begin{aligned}
	I_p(x) \geq (1-o(1)) (x-p)^2 I_p(1) \,, \qquad \forall x \in [0,1]
	\end{aligned}
	\end{equation}
	which applies for any $p=o(1)$. Indeed, we get \eqref{eq:50} by combining the elementary inequality
	$I_p(p-x) \geq I_p(p+x)$ for $0 \le x \le p \le \frac{1}{2}$ (cf. \cite[Lemma 3.3]{bg}), with the bound 
	\[
	\lim_{p \to 0} \inf_{x \in (0,1-p]} \Big\{ \frac{I_p(p+x)}{x^2 I_p(1)} \Big\} = 1 
	\]  
	of \cite[Corollary 3.5]{lz2}. Recall that $X\in \XX^d_n$ is symmetric, of non-negative entries with
	$\sum_{i=1}^n X_{ij} = d$ and $X_{jj}=0$ for all $j$. In particular, all eigenvalues 
	$\{\lambda_i\}$ of $X \in \XX^d_n$ are in $[-d,d]$, whereas with $d=np$ and 
	$a_{n,p} = n^2 p^2 I_p(1)$,  we deduce from
	\eqref{eq:50} that 
	\begin{equation}\label{eq:lbd-entropy}
	\begin{aligned}
	I_p(X) &\geq (1-o(1)) \sum\limits_{1 \le i \ne j \le n} (X_{i,j}-p)^2 I_p(1) \\
	&\ge (1-o(1)) \Big( \sum\limits_{i,j=1}^n 
	X_{i,j}^2 -d^2 \Big) I_p(1)  
	= (1-o(1))  a_{n,p}  \Big( \sum_{i=1}^{n} (\lambda_i/d)^2  - 1 \Big) \,.
	\end{aligned} 
	\end{equation}
	Further, $\mathrm{hom}(\sfC_l,X)= (n p)^{-l} \sum\limits_{i=1}^{n} \lambda^l_i$
	for any $X \in \XX_n$ and $l \ge 3$, so re-scaling 
	$\eta_i :=\lambda_i/(np)$, it follows from \eqref{eq:lbd-entropy} that 
	\begin{equation}\label{eq:basic-eig-bd}
	\begin{aligned}
	\phi^d_n(\sfC_l,1+\delta) & \geq \frac{1}{2}(1-o(1)) 
	\inf\{\sum\limits_{i=1}^{n} \eta^2_i-1\,  : \, |\eta_i| \leq 1, \sum\limits_{i=1}^{n} \eta^l_i \geq 1+\delta \}.
	\end{aligned} 
	\end{equation}
	The optimal $\{\eta_i\}$ in \eqref{eq:basic-eig-bd} are non-negative, so the desired bound 
	$\phi_n^d(\sfC_l,1+\delta) \ge (1-o(1)) c^d(\sfC_l,\delta)$ follows from considering our next lemma
	at $x_i = \eta_i^l \in [0,1]$, $\beta = 2/l$ and $\theta=1+\delta$.
	\begin{lemma}\label{th: realbd}
		For any $\beta \in (0,1)$ and $\uu{x} \in [0,1]^{\N}$, let $f_\beta(\uu{x}):= \sum_{i} x_i^{\beta}$.
		Then, for any $\theta \geq 0$,
		\[
		f_1(\uu{x}) \geq \theta \qquad \Longrightarrow \qquad 
		f_\beta(\uu{x}) \geq \lfloor \theta \rfloor + \{\theta\}^{\beta} \,.
		\]
	\end{lemma}
	\begin{proof} Since $f_\beta(\uu{x})$ is increasing in each coordinate, its infimum 
		over $K_{\ge \theta} := [0,1]^{\N} \cap \{ \uu{x} : f_1(\uu{x}) \ge \theta \}$, is attained at
		the convex set $K_{=\theta}$. Further, with $f_\beta(\cdot)$ a strictly concave function 
		(as $\beta \in (0,1)$), its infimum over $K_{=\theta}$ is attained at an extreme point of 
		$K_{=\theta}$, namely when all but at most one of the coordinates of $\uu{x}$ are 
		$\{0,1\}$-valued, and the stated lower bound immediately follows.
	\end{proof}
	%The only other connected graphs of maximal degree $\Delta=2$ are the path $\Pi_l$ of length $l \ge 1$
	% for which $\mathrm{hom}(\Pi_l,X)=n^{-l} p^{-(l-1)} \langle {\bf 1}, X^{l-1} {\bf 1} \rangle=1$
	%at all $X \in \XX_n^d$. Thus, obviously $\phi^d_n(\Pi_l,1+\delta)=\infty$ whenever $\delta>0$ and for 
	%all $l \ge 1$.
	
	\medskip
	\noindent {\bf Step 3} (Lower bound $\Delta \geq 3$):  We lower bound \eqref{eq:2} by 
	viewing $\XX_n^d$ as a subset of the collection $\WW$ of all \textit{graphons}
	(i.e. symmetric measurable $W:[0,1]^2 \rightarrow [0,1]$), via the map 
	% embedding  
	$W_X(s,t):=X_{[ns],[nt]}$. Indeed, doing so yields the bound
	$\Phi^{d}_n(H,u) \ge  n^2 \Phi^{d}(H,u)$ for the continuous problem 
	\begin{align}\label{dfn:Phi-d-cont}
	\Phi^{d}(H,u) := \frac{1}{2} \inf \{ I_p(W) :  & \, W \in \mathcal{W}, \; p^{-\edges(H)} t(H,W) \geq u, \;
	 \nonumber \\
	& \int_{0}^{1} W(x,y) dy = p  \; \textrm{ }\forall x \in [0,1] \},
	\end{align}\label{dfn:dens-graphon}
	where  $I_p(W) := \iint I_p(W(t,s)) dt ds$ denotes the entropy of graphon $W$ and 
	\begin{equation}\label{dfn:t-graphon}
	t(H,W) \coloneqq \int_{[0,1]^{\verts(H)}} \prod_{(i,j)\in E(H)} W(x_i,x_j) \prod\limits_{i=1}^{\verts(H)} dx_i 
	\end{equation}
	its homomorphism density. Next, as in   
	%the strategy of 
	\cite{bglz}, we change
	variables to $U \coloneqq W-p \in [-p,1-p]$, so our extra linear constraint translates to 
	\begin{equation}\label{eq:zero-deg}
	d(x) := \int_{0}^{1} U(x,y) dy = 0, \qquad  \forall x \in [0,1] \,.
	\end{equation}
	Using the standard notation $a_p \lesssim b_p$ whenever $a_p/b_p$ is bounded above as $p \rightarrow 0$, 
	in view of our scale $a_{n,p}$ it suffices for the lower bound on $\Phi^{d}(H,1+\delta)$
	to consider only $U$ such that
	\begin{equation}\label{eq:ipup}
	I_p(p+|U|) \leq I_p(p+U) \lesssim p^\Delta \log(1/p),
	\end{equation}
	with our next lemma thus the key 
	%  (which we prove in the sequel), 
	to the lower bound.
	\begin{lemma}\label{lemma: small}
		Suppose $H$ is a connected graph with maximal degree  $\Delta \geq 3$. If for $p \to 0$ the symmetric 
		$U:[0,1]^2 \rightarrow [-p,1-p]$ satisfy  \eqref{eq:zero-deg}, then 
		\begin{equation}\label{eq:remove-p}
		p^{-\edges(H)} t(H,p+U)= 1 + p^{-\edges(H)} t(H,U)+o(1)
		\end{equation}
	and for irregular $H$ also 
	\begin{equation}\label{eq:irregular-neg}
    p^{-\edges(H)} t(H,U)=o(1).
    \end{equation}
	% , yielding $p^{-\edges(H)} t(H,p+U)= 1 +o(1)$.
	\end{lemma}
    \noindent For irregular $H$ we have from \Cref{lemma: small} that $t(H,p+U)= (1+o(1))p^{\edges(H)}$
    whenever $U$ satisfies \eqref{eq:zero-deg} and hence, that 
    $\Phi^d(H,1+\delta) \rightarrow \infty$ as $p \to 0$. It follows that $\Phi^d(H,1+\delta)/(p^\Delta \log(1/p)) \to \infty$, so with $\Phi^{d}_n(H,u) \ge  n^2 \Phi^{d}(H,u)$ and rescaling,
    we deduce that
     $\phi_n^d(H,1+\delta) \to \infty$ as $p(n) \to 0$ (with
    $\delta>0$ fixed). Turning to deal with $\Delta$-regular $H$, we denote	
	by $\| \cdot \|_q$ the $L^q([0,1]^2)$-norms and recall that 
	for $|U| \le 1$ and graph $F$ of maximal degree  $\Delta(F) \ge 2$, 
	% applying 
	the generalized H\"{o}lder's inequality of \cite[Theorem 2.1]{gholder}
	for $\verts(F)$ variables and power $\Delta(F)$ at each $e \in E(F)$, yields that 
	\begin{align}\label{eq:lbd3}
	|t(F,U)| \le \|U\|_{\Delta(F)}^{\edges(F)} \le \|U\|_2^{2 \edges(F)/\Delta(F)} 
	\end{align}
	(see also \cite[Corollary 3.2]{lz1}).
	Thus, combining \eqref{eq:lbd3}
	and \Cref{lemma: small} we see that for
	any $W=p+U$ which is relevant for the \abbr{rhs} of \eqref{dfn:Phi-d-cont} at $\theta=1+\delta$,  
	we must have
	\[
	%  \int_{[0,1]^2} U^2(x,y) dx dy
	\|U\|_2^2  \geq  |t(H,U)|^{\Delta/\edges(H)} 
	\ge \big[ (\delta-o(1)) p^{\edges(H)} \big]^{\Delta/\edges(H)}  \,.  
	\]
	This, together with  \eqref{eq:50},  having $a_{n,p}=n^2 p^{\Delta} I_p(1)$ and  
	$\Delta \verts(H) = 2 \edges(H)$ (for $\Delta$-regular $H$), yield the required lower bound 
	\[
	n^2 \Phi^{d}(H,1+\delta) \ge \frac{1}{2} \delta^{2/\verts(H)} 
	(1+o(1)) a_{n,p} \,.
	\]
	We thereby proceed to complete the proof of \Cref{thm:regvar}, by proving
	\Cref{lemma: small}.
	\begin{proof}[Proof of \Cref{lemma: small}] Our starting point is the decomposition \cite[(6.1)]{bglz},
		\begin{equation}\label{eq:expand}
		p^{-\edges(H)} t(H,W) - 1= p^{-\edges(H)} [t(H,p+U)-t(H,p)] =\sum_F \NN(F,H) p^{-\edges(F)} t(F,U),
		\end{equation}
		over non-empty sub-graphs $F \subseteq H$, upto isomorphism, with $\NN (F,H)$ counting the number of sub-graphs of $H$ isomorphic to $F$. Further, recall \cite[(4.5)]{bglz} that \eqref{eq:ipup} implies in turn 
		\begin{equation}\label{eq:u2}
		\|U\|_2^2 \,
		%    = \int_{[0,1]^2} U^2(x,y) dxdy 
		\lesssim p^\Delta \,,
		\end{equation}
		whereas for $\Delta \ge 3$ and $U \ge 0$ it is shown in \cite[Corollary 6.2]{bglz}
		that under 
		% the constraint
		\eqref{eq:u2},
		any contribution to the \abbr{RHS} of \eqref{eq:expand} which is non-negligible 
		when $p \to 0$, must come from  
		\begin{equation}
		\FF_H \coloneqq \{F \subset H,  \;  \textrm{ with
			minimum vertex cover size } \;\; \tau(F)= \edges(F)/\Delta \},
		\end{equation}
		or from $F=H$ a $\Delta$-regular graph.
		We remove the restriction to $U\ge 0$,  by noting that for $\Delta \ge 3$ 
		the proof of \cite[Lemma 6.4]{bglz} applies to $|U|$ and since 
		$t(F,U) \le t(F,|U|)$, it follows that for $F \notin \FF_H$ which is not $\Delta$-regular,  
		\begin{equation}\label{eq:fh}
		t(F,U) = o(p^{\edges(F)} ) \,.
		\end{equation}
		We thus complete the proof of the lemma by showing that \eqref{eq:zero-deg} extends 
		the scope of \eqref{eq:fh} to every $F \in \FF_H$ which is not $\Delta$-regular. 
		Specifically, splitting $U=U^+-U^-$ to its positive  
		and negative parts $U^+ \in [0,1-p]$, $ U^- \in [0,p]$ induces the split 
		$d(x)=d^+(x) - d^-(x)$, where thanks to \eqref{eq:zero-deg},
		\[
		d^+(x) := \int_0^1 U^+(x,y) dy = d^-(x) :=  \int_0^1 U^-(x,y) dy \quad \textrm{ are in } \; 
		[0,p] \; \textrm{ for all }  x \in [0,1].
		\]
		In particular, by Jensen's inequality and \cite[Lemma 3.4]{lz2}, we have for $p \le p_o$ and 
		$b \in [p,1/3]$, that 
		\begin{equation}\label{eq:ipd}
		I_p(p+U) \geq I_p(p+|U|) \geq \int_0^1 I_p(p+2d^+(x))dx \geq \frac{I_p(p+2b)}{b^2} 
		\int_0^1 d^+(x)^2 dx \,.
		\end{equation}
		Recall that $I_p(p+2b) \sim 2b \log(2b/p)$ when $p^\alpha \le b \to 0$ with $\alpha<1$ fixed, 
		and deduce from  \eqref{eq:ipup} and \eqref{eq:ipd} that then
		\begin{equation}\label{eq:dbound}
		\int_{0}^{1} d^+(x)^2 dx \leq \frac{b^2 I_p(p+U)}{I_p(p+2b)} \lesssim p^\Delta b \,.
		\end{equation}
         Let $\SS \subset \{ \pm \}^{E(F)}$ enumerate those ${\bf s} \in \{ \pm \}^{E(F)}$, 
		with even number of minus entries. Then, setting  	
		\[
		{\bf U^s}(\uu{x}|F) := \prod_{e=(e_1,e_2) \in E(F)} \, U^{s_e} (x_{e_1},x_{e_2}) \,,
		\]
		we have that for any graph $F$
		\begin{equation}
		t(F,U)= t(F,U^+-U^-) \leq \sum_{{\bf s} \in \SS} t(F,{\bf U^s}),
		\quad 
		t(F,{\bf U^s}) :=
		\int_{[0,1]^{\verts(F)}} {\bf U^s} (\uu{x}|F) \prod_{i=1}^{\verts(F)} dx_i \,.
		\end{equation}
		 Adapting \cite[Lemma 7.4]{bglz} to our setting, we next 
		show that  
		\begin{equation}\label{eq:t-F-small}
		t(F,{\bf U^s}) =o(p^{\edges(F)}) \,,
		\end{equation}
		for any ${\bf s} \in \{\pm\}^{E(F)}$ and every 
		connected irregular bipartite $F$ of maximal degree $\Delta \ge 3$ such that $\tau(F)=\edges(F)/\Delta$.
		Indeed, \cite[Lemma 7.1]{bglz} shows that such $F$ contains a sub-graph $M$ of $\edges(M) = 2\tau(F)$ 
		edges, whose connected components $M_1, \ldots, M_k$ are path or even cycles, 
		with at least $M_1$ being a path of length $l \ge 1$. Since $\Delta(M_i)=2$ and the 
		bound  \eqref{eq:u2} applies also for $U^\pm$, it follows by \eqref{eq:lbd3} that for any choice of ${\bf s}$,
		\[
		|t(M_i,{\bf U^s})|  \le \| U^\pm \|_2^{\edges(M_i)}
		\lesssim \big( p^{\Delta} \big)^{\edges(M_i)/2} \,, \qquad i=1,\ldots,k \,.
		\] 
		Clearly, then 
		\[
		t(F,{\bf U^s}) \le t(M,{\bf U^s}) = \prod_{i=1}^k t(M_i,{\bf U^s})  \lesssim p^{\Delta \edges(M)/2}
		= p^{\Delta \tau(F)} = p^{\edges(F)} \,.
		\]
		For $M_1
		%=(x_1,\ldots,x_{l+1})
		$ a path of length $l = \edges(M_1) \ge 1$,  the 
		generalized H\"older's inequality yields the sharper bound 
		\begin{align*}
		t(M_1,{\bf U^s}) &= 
		% \int_{[0,1]^{l+1}} {\bf U^s}(\uu{x}|M_1) \prod_{i=1}^{l+1} dx_i  =  
		\int_{[0,1]^l} d^{\pm}(x_2)dx_2  \prod_{i=2}^l U^{s_{(i,i+1)}} (x_i,x_{i+1}) dx_{i+1} \\
		& \leq \|d^{\pm}\|_2 \|U^{\pm}\|_2^{l-1} \lesssim
		(p^\Delta b)^{1/2} p^{\Delta (l-1)/2} =o(p^{\Delta l/2}) \,,
		\end{align*}
		where the last inequality uses \eqref{eq:u2} and \eqref{eq:dbound} with $b \to 0$. 
		This establishes \eqref{eq:t-F-small},
		and consequently also \eqref{eq:fh},
		 whenever the irregular bipartite $F$ 
		of maximal degree $\Delta \ge 3$ such that $\tau(F)=\edges(F)/\Delta$, is connected.
		In particular, this applies for any connected $F \in \FF_H$ which is not $\Delta$-regular
		(see \cite[Section 6.2]{bglz}). 
         For non-connected $F \in \FF_H$ we first integrate out all isolated vertices of $F$ 		
		without altering the value of $t(F,U)$, and complete the proof by noting that 
		thereafter each connected component $F'$ of $F$ 
		must be in $\FF_H$ (where obviously $F' \ne H$ can not be $\Delta$-regular).
		Indeed, the non-empty
	     independent sets $S$ of $H^\star$ are in one-to-one 
	     correspondence with $F \in \FF_H$ which consists
	     of all edges of $H$ incident to $S$ (so
	     % In particular, $(S, V(F) \backslash S)$ forms a vertex bipartition of $F$, 
	     vertices of $S$ have degree $\Delta$ in $F$). 
	     Having all isolated vertices removed from $F$, each connected component $F'$ of $F$
	     must consist of all edges of $H$ incident to some non-empty subset of $S$. Hence,
	     $F' \in \FF_H$ as claimed.
	\end{proof}

	\subsection{Proof of \Cref{thm:regprob}}
	Setting hereafter $n_\edges := {n \choose 2}$ and $\KK_n^{(m)} := \{ \edges(G_n) = m \}$,
	recall Pittel's inequality (cf. \cite[(1.6)]{jbook}), that for $p=m/n_\edges$,
	any $n$, $m$ and event $\AA_n$
\begin{equation}\label{eq:ubd-unif-Gn}
% \bP_p(\AA_n) \ge 
\bP_p(\AA_n \cap \KK^{(m)}_n) = \bP_p(\KK^{(m)}_n) \bP^{(m)}(\AA_n) \ge \frac{1}{3 \sqrt{m}} \, \bP^{(m)}(\AA_n)
\end{equation}
where as before, $\bP_p$ denotes the law of the \abbr{er}-model $\GG(n,p)$. 
Further, under
     \eqref{eq:p-range} we have that $a_{n,p} \gg \log n \ge \frac{1}{2} \log m$, so we 
     get the bound  \eqref{eq:3} by 
     combining \eqref{eq:ubd-unif-Gn} for $\AA_n=\{ \mathrm{hom}(H,G_n) \geq t \}$ 
     with \cite[Thm. 1.2 \& Thm. 1.3]{cod}
     (see \Cref{rem:samot} on its improvement when \cite[Thm. 1.5]{samotij} applies).
	
	A similar, but more delicate argument yields \eqref{eq:4}. Specifically, similarly to \cite{cod}, 
	we view the \abbr{ER}-law $\bP_p$ of
$A_{G_n} \in \XX_n$ as the product Bernoulli measure $\mu_p$
of a random binary vector $\uu{x} \in [0,1]^{n_\edges}$ (namely, the upper-triangular part
of $A_{G_n}$). Then, by an intersection with a given closed, convex $K \subset [0,1]^{n_\edges}$
one easily extends the non-asymptotic bound of  \cite[Corollary 2.2]{cod} 
to get for any $h: [0,1]^{n_\edges} \rightarrow \R^k$, $k \ge 1$, $\bm{p} \in [0,1]^{n_\edges}$, $t \in \R$ and 
$\delta>0$ that
\begin{equation}\label{eq:ip2}
\mu_{\bm{p}}(\{h \ge \uu{t} \} \cap K) \leq |\mathbb{I}| \exp\big(-\inf_{\stackrel{h(\uu{x}) \ge \uu{t}-\delta {\bf 1}}{\uu{x} \in K}} I_{\bm{p}}(\uu{x})\big) 	+ \mu_{\bm{p}}(\EE \cap K) \,,
\end{equation}
provided $K \cap \{0,1\}^{n_\edges} \setminus \EE$ is covered by a collection $\{B_i\}_{i \in \mathbb{I}}$
of closed convex subsets of $[0,1]^{n_\edges}$ and
\begin{equation}\label{eq:h1}
\max_{i 	\in \mathbb{I}} \; \sup_{\uu{x},\uu{y} \in B_i \cap K} \| h(\uu{x})-h(\uu{y})\|_\infty \leq \delta \,.
\end{equation}
Recall that for fixed $t>1$ and $p$ as in \eqref{eq:p-range}, one arrives, as in  \cite[Thm. 1.2 \& 1.3]{cod},
at
\begin{equation}\label{eq:codmain}
a_{n,p}^{-1}\log\bP_p(\mathrm{hom}(H,G_n) \geq t) \leq -\phi_{n,p}(H,t-o(1))+o(1) \,,
\end{equation}
by applying \eqref{eq:ip2} for $\bm{p}=p \bm{1}$, 
$\R_+$-valued $h(\cdot)=\mathrm{Hom}(H,\cdot)$
and $K=[0,1]^{n_\edges}$, after excluding a 
%suitable
set $\EE=\EE_{\kappa,\delta}$ of $\bP_p$-probability 
%at most 
$\exp(-\kappa a_{n,p})$ 
for arbitrarily large $\kappa$ (cf. \cite[(3.12),(4.5),(6.13),(6.23)]{cod}).
The bulk of the work there is a deterministic analysis to exhibit a cover of $\{0,1\}^{n_\edges} \setminus \EE$ by 
$\exp(o(a_{n,p}))$ many closed convex sets $\{B_i\}$ that satisfies \eqref{eq:h1} 
for arbitrarily small $\delta> 0$. The same reasoning, now with closed, convex $K_n$ 
which is the upper-triangular image of $\XX_n^d$ for $d=n p$, and using the same $\{B_i\}$, $\EE$ 
as in the proof of \cite[Thm. 1.1 \& 1.2]{cod}, yields that for $\KK^d_n := \{ A_{G_n} \in \XX_n^d \}$
and some $\kappa_n \uparrow \infty$,
 \begin{equation}\label{eq:codomain2}
a_{n,p}^{-1}\log\bP_p(\{\mathrm{hom}(H,G_n \geq t\} \cap \KK^d_n) \leq -(\phi^d_n (H,t-o(1))-o(1))\wedge \kappa_n \,
\end{equation}
(keeping the $\kappa_n$ term in the absence of a uniform upper bound such as \cite[(6.1)]{cod} 
on $\phi^d_n (H,t)$).
Hence, applying the well known identity
\begin{equation}\label{eq:11}
\bP_p(\AA_n \cap \KK^d_n) = \bP_p(\KK^d_n) \bP^{d}(\AA_n),
\end{equation}
in the special case of $\AA_n=\{\mathrm{hom}(H,G_n) \geq t\}$, we complete the proof of \eqref{eq:4} upon 
showing that for $d=n p$ and $p(n) \to 0$ of \eqref{eq:p-range},
\begin{equation}\label{eq:as-countreg}
a_{n,p}^{-1} \log \bP_p(\KK^d_n) \to 0 \,.
\end{equation}
To this end, with $g_n(\textbf{d})$ denoting the number of simple graphs $G_n$ 
of degrees $\textbf{d}=(d_1,d_2,\ldots,d_n)$, recall that the $\bP_p$-probability 
of producing a graph of such degrees is precisely
\[
g_n(\textbf{d}) \, p^{n \bar d/2} (1-p)^{n_\edges - n \bar d/2}  \,,
\]
where $\bar d: =n^{-1} \sum_{i=1}^{n} d_i$ (and hereafter $n \bar d$ assumed even). 
We thus establish \eqref{eq:as-countreg} by utilizing the asymptotic count of \cite[Corollary 1.5]{wormald},
\begin{equation}\label{eq:wormald}
g_n(\textbf{d})\sim \sqrt{2} \exp \big(\frac{1}{4}-\frac{\gamma^2}{4\mu^2(1-\mu)^2} \big) \left(\mu^{\mu} (1-\mu)^{(1-\mu)}\right)^{n_\edges} \prod_{i=1}^n \binom{n-1}{d_i} \,,
\end{equation}
where $\mu:=\bar d/(n-1)$, $\gamma:=(n-1)^{-2} \sum_{i=1}^{n}(d_i - \bar d)^2$, and
\eqref{eq:wormald} holds whenever $n (\bar d \wedge (n-1 - \bar d)) \to \infty$ and 
$n^{-\varepsilon} \max_j |d_j-\bar d|^2=o(\bar d \wedge (n-1-\bar d))$ for 
some fixed $\varepsilon >0$. In particular, this applies for $d_i = \bar d = n p$ 
and $n^{-1} \ll p \ll 1$ (as in \eqref{eq:p-range}), where $\mu = d/(n-1)$ and $\gamma=0$, 
resulting with 
\begin{align}\label{eq:countreg}
n^{-1} \log \bP_p(\KK^d_n) &\ge d \log \mu + (n-1-d) \log (1-\mu) + \log {n-1 \choose d} + o(1) 
\nonumber \\
& \sim - \frac{1}{2} \log d
\end{align}
(using Stirling's formula in the last step). With $d=np$ and $n \log (n p) = o(a_{n,p})$ in the 
range assumed in \eqref{eq:p-range}, this implies the limit \eqref{eq:as-countreg}, thereby
completing the proof. 
\qed
	
From \eqref{eq:wormald} we further deduce the following estimate, which we later use  
in proving \Cref{th:correg+unif}.
\begin{lemma}\label{lem:d-clique-d-reg}
For any $\vep>0$, $1 \gg p \ge \vep n^{-1/2}$, integer $d= n p$ and $n'/n \to 1$,  one has that 
\begin{equation}\label{eq:max-clique}
\frac{g_{n'-d-1}(d {\bf 1})}{g_{n'}(d {\bf 1})} \ge p^{d_\edges (1+o(1))} \,. 
\end{equation}
\end{lemma}	
\begin{proof} For $\mu=\frac{d}{n'-1} = O(p)$, we get from \eqref{eq:wormald} 
by Stirling's formula, similarly to \eqref{eq:countreg}, that 
\[
\frac{2}{n'} \log g_{n'} (d {\bf 1}) = d \log (n'-1) - d  \log d + d -  \log d + O\big( d^2 n^{-1} \big) \,.
\]
The same applies at $\widehat{n}=n'-d-1$,  resulting with 
\begin{align*}
\log g_{\widehat{n}} (d {\bf 1}) - \log g_{n'} (d {\bf 1}) &= \frac{d}{2} \Big[
\widehat{n} \log (\widehat{n}-1) - n' \log (n'-1) + d \log d  + O(d) \Big] \\
&=  d_{\edges} [ \log \mu  + O(1) ]  = d_{\edges} (1+o(1)) \log p 
\end{align*}
as claimed (using in the last step that $\log (1/\mu) = O(d)$ while $p \to 0$ and $n'/n \to 1$).
\end{proof}	

Next, we prove a combinatorial estimate which we later employ in proving \Cref{th:correg+unif}.
\begin{lemma}\label{lemma: sdkvcount}
	Fixing $\vep >0$, if $d=n p \ge 2$ and a collection $\sfE_n$ of edges
	on $[n]$ has maximal degree $s_1 < (1-\vep) d$, 
	then for any $uw \in \sfE_n$
	\begin{equation}
	\bP(\sfE_n \subseteq G_n^d) \geq 
	\vep^2 p (\vep - 7 p) \, \bP(\sfE_n \textbackslash \{uw\} \subseteq G_n^d)   \,.
	\end{equation}
\end{lemma}
\begin{proof} Let $\CC_1$ denote the collection of $d$-regular graphs 
	containing all of $\sfE_n$, with $\CC_0$ the collection of $d$-regular graphs 
	containing $\sfE_n \textbackslash \{uw\}$, but not the edge $uw$.  Since
	\begin{equation*}
	\bP(\sfE_n \subseteq G_n^d)= \frac{|\CC_1|}{|\CC_1|+|\CC_0|} \, \bP(\sfE_n \textbackslash \{uw\} \subseteq G_n^d) \,,
	\end{equation*}	
	it suffices to show that for all $n$ 
	\begin{equation}\label{eq:C1-C0-ratio}
	\frac{|\CC_1|}{|\CC_0|} \ge (d-s_1)(d-s_1-1) \frac{\vep nd - 6 d^2}{(nd)^2} \,.
	\end{equation}
	To this end, recall as in the proof of \cite[Lemma 2.3]{ksv}
	that for $G \in \CC_1$,  any pair of edges $u_i w_i \in G \setminus \sfE_n$, $i=1,2$,
	with disjoint vertices $\{u,w,u_i,w_i\}$, such that the triplet 
	$\sfS' := \{ w u_1, w_1 u_2, w_2 u \}$ is disjoint of $G$, defines 
	a forward switching, where replacing $\sfS:=\{uw,u_1w_1,u_2w_2\}$ by $\sfS'$ 
	results with $G' \in \CC_0$. Conversely, per $G' \in \CC_0$, any  
	disjoint $\{u,w,u_i,w_i\}$ such that $\sfS'$ is in $G' \setminus \sfE_n$ 
	while $\sfS$ is disjoint of $G'$, provides a reverse switching where replacing 
	$\sfS'$ by $\sfS$ recovers a graph $G \in \CC_1$.  A double counting argument 
	bounds $|\CC_1|/|\CC_0|$ below by the minimum over $G,G'$ of the 
	ratio between the number of reverse switching and the number of forward switching. 
	Counting edges \abbr{wlog} as oriented, a $d$-regular graph $G$ has at most $nd$ edges, 
	so the number of forward switching never exceeds $(n d)^2$. As for the reverse switching,
	given $u \ne w$, we have at least $d-s_1$ choices for $u_1 \notin \{u,w\}$ such that
	$w u_1 \in G' \setminus \sfE_n$ and $(d-s_1-1)$ choices of $w_2 \notin \{u,w,u_1\}$
	such that $w_2 u \in G' \setminus \sfE_n$. There are further at least $n-2-2d$ vertices beyond
	$\{w,u,u_1,w_2\}$ which are neither connected by $G'$ to $u_1$ nor to $w_2$. Within 
	those vertices there are at least $d (n- 2(2+2d))$ possible edges $w_1 u_2 \in G'$, at most 
	$n s_1 < (1-\vep) n d$ of which are from $\sfE_n$.  With the number of reverse switching per $G'$ thus being
	at least $(d-s_1)(d-s_1-1)d(\vep n-4-4d)$, we arrive at 
	\eqref{eq:C1-C0-ratio}, as claimed.
\end{proof}	
	
\subsection{Proof of \Cref{th:correg+unif}} Fix hereafter $H$ and $p(n)$ as in 
\Cref{thm:regprob} (replacing first $H$ by its $2$-core if considering \eqref{eq:8}).
The
stated lower bounds on the limits \eqref{eq:8} and \eqref{eq:7}, then follow by combining 
\Cref{thm:regprob}  with \Cref{thm:regvar} and \eqref{eq:ervar}, 
respectively. \par 
To prove \eqref{eq:8}, recall that $c^d(H,\delta)=\infty$ for irregular $H$, in which case the complementary upper bound in \eqref{eq:8} trivially holds. Further, in view of \eqref{eq:11}
we get the stated upper bound for connected $\Delta$-regular $H$, upon showing that for any $\delta' < \delta$ 
\begin{align}\label{eq:prob-d-lbd}
\liminf_{n \to \infty} a_{n,p}^{-1} \, \log \bP_p(\{\mathrm{hom}(H,G_n) \geq 1+\delta' \} \cap \KK^d_n) &\ge - c^d(H,\delta) \,,
\end{align}
where $\KK_n^d = \{ A_{G_n} \in \XX_n^d \}$. We derive \eqref{eq:prob-d-lbd} by a change of measure 
to a inhomogeneous \abbr{er}-model, denoted hereafter $\bP_\star$, where the  edge probabilities are set via $X_n^\star \in \XX_n^d$ of a block form. Specifically, we split the proof 
of \eqref{eq:prob-d-lbd} into {\bf Case 1} where $H=\sfC_l$, $\Delta=2$ with
$X_n^\star$ taken as in
 \eqref{eq:opt-cyc}, and {\bf Case 2} which corresponds to $\Delta$-regular 
$H$, $\Delta \ge 3$, using then $X_n^\star$ of \eqref{eq:opt-clique}. 
Denoting by $\sfE_n^\star$ the collection of edges within the ${\bf 1}$-blocks of 
$X_n^\star$ of \eqref{eq:opt-cyc} and \eqref{eq:opt-clique}, respectively, in both cases  
$\bP_\star$ is supported on the event
$\AA_n = \{\sfE^\star_n \subseteq G_n\}$. Further, thanks to the specific block structure 
of $X_n^\star$, the constraint $\AA_n \cap \KK_n^d$, imposes
the ${\bf 0}$-blocks (in case of \eqref{eq:opt-cyc}), and a
\emph{non-random} number of edges per block in $A_{G_n}$ 
(which matches that block's aggregate in $X_n^\star$). Thereby,
at every instance within $\AA_n \cap \KK_n^d$ we have the Radon-Nikodym derivative 
\begin{equation}\label{eq:rnd-val}
\frac{d\bP_p}{d\bP_\star} = e^{-\frac{1}{2} I_p(X_n^\star)} \,.
\end{equation}
While proving \Cref{opt-reg} we saw that $I_p(X_n^\star) \le 2 c^d(H,\delta) a_{n,p} 
+ o(a_{n,p})$ both for \eqref{eq:opt-cyc} ($\Delta=2$), and for \eqref{eq:opt-clique}
($\Delta \ge 3$). In conclusion, upon intersecting the event 
in \eqref{eq:prob-d-lbd} with $\AA_n$, followed by
such change of measure, it suffices to show that in both cases, 
\begin{equation}\label{eq:star-lln}
\bP_\star (\mathrm{hom}(H,G_n) < 1+\delta') \ll \bP_\star( \AA_n \cap \KK_n^d) = e^{- o(a_{n,p})},
\end{equation}
 and thereby having that
\[
a_{n,p}^{-1} \log \bP_\star (\{\mathrm{hom}(H,G_n) \geq 1+\delta'\} \cap \KK_n^d 
\cap \AA_n)\to 0 \,.
\]
Proceeding to establish \eqref{eq:star-lln},  denoting hereafter by $\wt{n}:= n-s$ 
the size of the (bottom) ${\bf q}$-block, recall that in both cases
$\wt{n}/n \to 1$ and $q/p \to 1$. In {\bf Case 1} the non-random
contribution to $\mathrm{hom}(H,G_n)$ from the planted cliques under $\bP_\star$ 
is $\delta (1+ o(1))$ (see \eqref{eq:tH-cyc}). The same applies for the single 
non-random clique planted in {\bf Case 2}, so in both cases
%  {\bf Case 1} and {\bf Case 2} 
we have the following upper bound for any fixed $\vep < \delta-\delta'$ and $n$ large enough 
\begin{equation}\label{eq:move-to-q}
\bP_\star (\mathrm{hom}(H,G_n) < 1+\delta') \le \bP_q(\mathrm{hom}(H,G_{\wt{n}}) < 1-\vep) \,.
\end{equation}
Now, we want to invoke upper bound of \cite[Theorem 3]{ltail} on the lower tail for homomorphism counts to the \abbr{rhs} of  \eqref{eq:move-to-q}. To this end, note that $(4-\verts(J)) \edges(J) \le \verts(J)$ for any graph $J$ with at least one edge ($\edges(J) \le 3$ for $\verts(J)=3$, $\edges(J) \le 1$ for $\verts(J)=2$).  Thus, when
$\edges(J) \ge  1$ and $\Delta(J) \le \Delta$, the bound $2 \edges(J) \le \Delta \verts(J)$ 
implies that  $(\edges(J) -1) \le \Delta (\verts(J) -2)$. In particular, our condition
\eqref{eq:p-range} of $p \gg n^{-1/\Delta}$ implies that for any $J \subseteq H$ with $\edges(J) \ge 1$
\[
\bE_q [\mathrm{Hom}(J,G_{\wt{n}})] = \wt{n}^{\verts(J)} q^{\edges(J)} \ge (1+o(1)) n^2 p \,.
\]
Hence, by \cite[Theorem 3]{ltail}, for any such $J$, 
$q/p \to 1$, $p \gg n^{-1/\Delta}$ and all 
$n$ large enough, 
\begin{align}\label{eq:lbd-J-hom}
\bP_q(\mathrm{hom}(J,G_{\wt{n}}) < 1-\vep)
 &\le \exp \big(-\Theta(\vep^2 \min_{\Delta(J) \le \Delta, \edges(J) \ge 1} \{ n^{\verts(J)} p^{\edges(J)} \} ) \big) 
\nonumber \\
& \le \exp \big( - \Theta (\vep^2 n^2 p)) \ll \exp(- O(a_{n,p}) ) \,.
\end{align}
Considering $J=H$, it thus follows from \eqref{eq:move-to-q}-\eqref{eq:lbd-J-hom} that
\eqref{eq:star-lln} holds as soon as 
\begin{equation}\label{eq:star-regular}
a_{n,p}^{-1} \log \bP_\star( \AA_n \cap \KK_n^d) \to 0 \,.
\end{equation}
To this end, note that combining \eqref{eq:11} and \eqref{eq:rnd-val} we arrive at the identity
\begin{equation}\label{eq:ident-lbd}
\bP_\star(\AA_n \cap \KK_n^d) = e^{\frac{1}{2} I_p(X_n^\star)} \bP_p(\AA_n \cap \KK_n^d) 
= e^{\frac{1}{2} I_p(X_n^\star)} \bP_p(\KK_n^d) \bP(\sfE^\star_n \subseteq G_n^d)  \,.
\end{equation}
Recall that $a_{n,p} \, c^d(H,\delta) = |\sfE^\star_n| \log(1/p) (1+o(1))$. Further,  
while proving \Cref{opt-reg}  we saw that 
$a_{n,p}^{-1} I_p(X_n^\star) \to 2 c^d(H,\delta)$.
Thus, by \eqref{eq:as-countreg} and \eqref{eq:ident-lbd} we get \eqref{eq:star-regular} 
once we show that
\begin{equation}\label{eq:p-Es}
\bP(\sfE^\star_n \subseteq G_n^d) \ge p^{-|\sfE^\star_n| (1+o(1))}  \,.
\end{equation}
For {\bf Case 2} which has a single clique $\sfE^\star_n$ of size $s_1=o(d)$, 
we get \eqref{eq:p-Es} by sequentially peeling its $(s_1)_\edges$ edges
and iteratively employing \Cref{lemma: sdkvcount} for the relevant subsets of $\sfE^\star_n$. 
Turning to {\bf Case 1}, the event $\{\sfE^\star_n \subseteq G_n^d\}$ is then the intersection of 
$\lceil \delta \rceil$ \emph{independent} events. These amount to having 
$\lfloor \delta \rfloor$ disjoint maximal cliques of size $d+1$ each, 
and for $\{ \delta \} \in (0,1)$ having an additional clique of size 
$s_1 < (1-\vep(\delta)) d$ within the remaining ($d$-regular) graph $G^d_{\wt{n}+s_1}$.
Since $\vep^3 p = p^{1+o(1)}$, the contribution of the 
latter $s_1$-sized clique to $\bP(\sfE^\star_n \subseteq G_n^d)$ is likewise handled 
by $(s_1)_\edges$ applications of \Cref{lemma: sdkvcount}. Further, 
the contribution of maximal cliques to that probability, is  
precisely the product of the \abbr{lhs} of \eqref{eq:max-clique} at
$n'=n-(d+1) j$ for $0 \le j < \lfloor \delta \rfloor$, which for
$\wt{n}=n-s \gg d$ yields the bound \eqref{eq:p-Es} also in {\bf Case 1}. \par

Turning next to the upper bound in \eqref{eq:7}, considering the identity in \eqref{eq:ubd-unif-Gn} and the characterization of $c(H,\delta)$ via
the \abbr{rhs} of \eqref{eq:multans} at $\sk=1$, this amounts to showing that for 
$\KK_n^{(m)} = \{ A_{G_n} \in \XX_n^{(m)} \}$ and connected
$H$ of maximal degree $\Delta \ge 2$,
\begin{align}
\liminf_{n \to \infty} a_{n,p}^{-1} \log \bP_p(\{\mathrm{hom}(H,G_n) \geq 1+\delta'\} \cap \KK^{(m)}_n) &\ge 
- x - \frac{1}{2} y^2 \,,
\label{eq:prob-m-lbd-hub}
\end{align}
provided $x,y > 0$ satisfy 
\begin{equation}\label{eq:xy-hom}
\sfP_{H^\star}(x)+ y^{\verts(H)} \, \ind_{\{ H\;\;  \mathrm{is} \;\; \Delta\textrm{-}\mathrm{regular} \}} \ge 1+ \delta \,.
\end{equation} 
Fixing $x,y \ge 0$ for which \eqref{eq:xy-hom} holds, we choose here as
$\bP_\star$ the inhomogeneous \abbr{er}-model whose edge probabilities are set by
the matrix $X_n^\star \in \XX_n^{(m)}$ of the form 
\begin{equation}\label{eq:clhub}
X_n^\star= 
\begin{bmatrix}
\textbf{1} & \mathbf{1} & \mathbf{1} \\ 
\textbf{1} & \mathbf{1} & \mathbf{q} \\
\textbf{1} & \textbf{q} & \textbf{q}
\end{bmatrix},
\end{equation} 
with principal blocks of (integer) sizes $s_1$, $s-s_1$, $n-s$, where
$s_1 \sim x p^\Delta n$, $s \sim y p^{\Delta/2} n$ (so $s_1 \ll s \ll n$),
and $q$ satisfies the global edge constraint 
\begin{equation}\label{eq:q-hub-m}
(1-q) [s_{\edges} + (n-s) s_1]  + q \, n_{\edges} = p \, n_{\edges} \,.
\end{equation}
As before, $\bP_\star$ is supported on $\AA_n$ (now for  the edges 
$\sfE_n^\star$ within the ${\bf 1}$-blocks of \eqref{eq:clhub}), and thanks to 
\eqref{eq:clhub}-\eqref{eq:q-hub-m}, the relation \eqref{eq:rnd-val} holds 
throughout $\AA_n \cap \KK_n^{(m)}$. The contribution to $\frac{1}{2} I_p(X_n^\star)$ from the
$(x+y^2/2) n^2 p^\Delta (1+o(1))$ entries in the 
${\bf 1}$-blocks of $X_n^\star$ of \eqref{eq:clhub}
matches the lower bound in \eqref{eq:prob-m-lbd-hub}.
Due to the constraint \eqref{eq:q-hub-m}, the value 
of all but $O(n^2 p^\Delta)$ entries of the latter matrix, is $q=p - O(p^\Delta)$, so
those (non ${\bf 1}$) 
entries have a cumulative $o(a_{n,p})$ effect on $I_p(X_n^\star)$ (see \eqref{eq:qpbound}). 
Thus, as in the preceding proof of \eqref{eq:8}, it 
suffices to show \eqref{eq:star-lln}, albeit with $\KK_n^{(m)}$ instead of $\KK_n^d$. 

To this end,
recall that, $\wt{n}/n \to 1$ and $q/p \to 1$. For $\Delta$-regular $H$, the clique of size $(s-s_1)$  planted in the middle of 
$X_n^\star$ contributes $y^{\verts(H)}(1+o(1))$ to $\mathrm{hom}(H,G_n)$  
(see \eqref{eq:exess-clique}, where 
$2 \edges(H) = \Delta \verts(H)$ for any $\Delta$-regular $H$). Further, by definition of $H^\star$, 
restricting $H$ to $S^c := V(H) \setminus S$ for an independent set $S$ of $H^\star$, 
yields a sub-graph $H_{S^c}$ of precisely $
%\edges(H_{S^c}) = 
\edges(H) - \Delta |S|$ edges.
From \eqref{eq:xy-hom} and the definition of 
%the independence polynomial 
$P_{H^\star}(\cdot)$ we thus deduce that
\begin{align*}
(1+\delta - y^{\verts(H)} \, \ind_{\{ H\;\;  \mathrm{is} \;\; \Delta\textrm{-}\mathrm{regular} \}})
n^{\verts(H)} p^{\edges(H)} &\le \sum_{\substack{S \textrm{ independent}\\ \textrm{set of } H^\star}} 
x^{|S|}  n^{\verts(H)} p^{\edges(H)} \\ 
& = (1+o(1)) 
\sum_{\substack{S \textrm{ independent}\\ \textrm{set of } H^\star}} 
s_1^{|S|} \, \wt{n}^{\verts(H_{S^c})} q^{\edges(H_{S^c})} \,,
\end{align*}
which in turn yields for any fixed $\vep < (\delta-\delta')/(1+\delta)$ and $n$ large enough the bound
\begin{equation}\label{eq:move-to-q2}
\bP_\star (\mathrm{hom}(H,G_n) < 1+\delta') \le
\sum_{\substack{S \textrm{ independent set}\\ \textrm{of } H^{\star}; \; e(H_{S^c}) \ge 1}} 
\bP_q(\mathrm{hom}(H_{S^c},G_{\wt{n}}) < 1-\vep) \,.
\end{equation}
Applying \eqref{eq:lbd-J-hom} for $J=H_{S^c}$ and enumerating over the independent sets $S$ 
on the \abbr{rhs} of \eqref{eq:move-to-q2}, we deduce that its \abbr{lhs} is bounded 
for all $n$ large enough by $\exp(- O(a_{n,p}) )$. Thus, as in the proof of \eqref{eq:star-lln}, it 
remains only to show that similarly to \eqref{eq:star-regular}, we have here that
\begin{equation}\label{eq:star-reg-new}
a_{n,p}^{-1} \log \bP_\star(\AA_n \cap \KK_n^{(m)}) \to 0 \,.
\end{equation}
Our choice of $q$ in \eqref{eq:q-hub-m} is 
such that under $\bP_\star$ the requirement to be in $\XX_n^{(m)}$ amounts to the number of edges in the 
${\bf q}$-block of size $L_q = O(n^2)$ in $X_n^\star$, matching its specified integer valued
mean $q L_q$. Further, $q/p \to 1$, hence $q L_q = O(m)$ and \eqref{eq:star-reg-new} then 
holds by  Pittel's inequality and \eqref{eq:p-range}.

	\section{Inhomogeneous graph ensembles}\label{sec:sbm}
	\subsection{Proof of \Cref{th:sbmprob}}
As in the proof of \Cref{thm:regprob}, our starting point towards proving \eqref{eq:sbm1}
is again \eqref{eq:ip2}, taking now  
$K=[0,1]^{n_\edges}$ and $\R_+$-valued $h(\cdot)=\mathrm{Hom}(H,\cdot)$, as in the 
proof of \cite[Thm. 1.1 \& 1.2]{cod}, while replacing the constant vector $p \bm{1}$ 
with the given $\bm{p}=(p_{ij})$ that corresponds to $\bP^{[\ell]}$. Note that our
assumption that $c_\infty := \max_{n,r,r'} \{ c^{(n)}_{r r'} \}$ is finite, with
$\alpha_o := \inf_n \{ \alpha^{(n)}_1 \}$ and
$c_o := \inf_n \{ c^{(n)}_{11} \}$ positive, guarantee that both $b_H^{(n)}$ and
$\phi_{n,\bm{p}}(H,t)/\phi_{n,p}(H,t)$ be bounded
away from zero and infinity, per fixed $H$ and $t >1$. In particular, up to 
some change of universal constants the uniform bounds 
\cite[(6.1) and (6.2)]{cod} apply to $\phi_{n,\bm{p}}(H,t)$.
Thus, using hereafter 
the \emph{same cover} of $\{0,1\}^{n_\edges} \setminus \EE$ by $\exp(o(a_{n,p}))$ many 
closed convex sets $\{B_i\}$ (which satisfy \eqref{eq:h1} for arbitrarily small $\delta>0$), 
as in the proof of \cite[Thm. 1.1 \& 1.2]{cod}, one needs only to verify that under our inhomogeneous \abbr{ER}-law:
%$\bP^{[\ell]}$, 
\newline
(a). For any $\gamma<\infty$ the exceptional sets 
$\EE_0(\kappa_0)$ from \cite[Prop. 3.4]{cod} and $\LL_F(\kappa_1)$ of 
\cite[(3.14)]{cod} for strictly induced sub-graphs $F \prec H$, are of probability
$\exp(-\gamma a_{n,p})$ when $\kappa_i(\gamma)$ are large enough.
\newline
(b).  For  $\kappa = \kappa_p  \gg \log (1/p)$ the event 
\begin{equation}\label{eq:GG-com}
\EE_0(\kappa):=\GG(\sqrt{\kappa} (n p), C' \sqrt{n p})^c \,,
\end{equation}
from \cite[(4.2)]{cod}, satisfies the bound \cite[(4.5)]{cod}.
\newline
To this end,  we replace the adjacency matrix ${\sf J}_n$ of the complete graph 
% ={\bf 1 1}^T- {\sf I}_n$
(see \cite[(1.40)]{cod}), by the weighted matrix 
${\sf J}_n = p(n)^{-1} {\bf p}$ of zero main diagonal and uniformly 
bounded above and below entries (in particular, with 
$n^{-1} \|{\sf J}_n \|_{\mathrm{HS}}$ uniformly bounded). Then, further
examining \cite{cod}, 
the preceding tail probability bounds (a) and (b) are direct consequences of the following analogs of 
\cite[(4.7)]{cod} and \cite[(6.13)]{cod}, for the inhomogeneous adjacency matrix 
$X_n:=A_{G^{[\ell]}_n}$.
\begin{lemma}\label{lemma:sbmtight}
For some $C'$ finite, $c_\star>0$, any $t \geq 0$, $p \ge \, (c_o n)^{-1} \log n$ and all $1\leq k \leq n$, we have
\begin{align}\label{eq:hs}
\bP^{[\ell]}(\|(X_n -p {\sf J}_n)_{\leq k}\|_{\mathrm{HS}} \geq  C' \sqrt{knp} + t) & \leq 4 e^{-t^2/16} \, , 
\\
\label{eq:homK11}
\bP^{[\ell]} ( {\bf 1}^T X_n {\bf 1} > \kappa n^2 p) & \le e^{-c_\star \kappa n^2 p}  \,.
\end{align}
\end{lemma}	
\begin{proof} Following  the proof of 
\cite[Lemma 4.3]{cod} it suffices for \eqref{eq:hs}
to show that 
%, similarly to \cite[Lemma 4.5]{cod}, that 
\begin{equation}\label{eq:hsmed}
\bE^{[\ell]} (\|X_n -p {\sf J}_n\|_{\mathrm{op}} ) \leq  \frac{C'}{2} \sqrt{n p}  \,.
\end{equation}
Recall
\cite[Example 4.10]{LHY}, that for some $c<\infty$, if
$d_n := \max\limits_i \{\sum\limits_j p_{ij} \} \ge \log n$, then 
\begin{equation}\label{eq:lhy}
\bE \big( \|X_n-\bE X_n \|_{\mathrm{op}} \big) \leq c \sqrt{d_n} 
\end{equation}
for any symmetric $n \times n$ 
matrix $X_n$ of independent Bernoulli($p_{ij}$) entries. For $X_n = A_{G_{n}^{[\ell]}}$
we have that $\bE X_n = p {\sf J}_n$  and 
$(np)^{-1} d_n  \in [c_o, c_\infty]$, yielding \eqref{eq:hsmed}. 
% for C' = 2 c \sqrt{c_\infty}
Next, \eqref{eq:homK11} follows from  \cite[(6.13)]{cod} 
% (at $c_\infty p$), 
upon
noting that replacing in $\bP^{[\ell]}$ the values $\{c^{(n)}_{rr'}\}$ by $c_\infty$ may only increase the \abbr{lhs} 
of \eqref{eq:homK11}.
\end{proof}

Turning next to the proof of \eqref{eq:sbm2}, note first that upon replacing 
$I_p(Q)$ by $\frac{1}{2} I_{\bm{p}}(Q)$ and $\wt{C}$ by $\wt{C}/c_o$, the bound
\cite[(6.5)]{cod} holds in our setting, and thereby so does \cite[(6.8)]{cod}, provided 
$p \le 1/(2 c_\infty)$.  As $1 \gg p$ and $n p^{\Delta(H)} \gg 1$, it follows that 
for some $\eta_{n,p} = o(a_{n,p})$ and all $X \in \XX_n$
\begin{equation}\label{eq:basic-lbd}
\bP^{[\ell]}(\BB_X) \ge \exp(-\frac{1}{2} I_{\bm{p}}(X) - \eta_{n,p}) \,, \quad
% \mathrm{for} 
\quad  \BB_X := \{ Z \in \XX_n : \|Z - X\|_{\mathrm{op}} \le C_0 \sqrt{n} \} \,.
\end{equation}
Fixing $t>1$, $\vep>0$, upon combining \eqref{eq:sbm1} with the lower bound 
of the form of \cite[(6.2)]{cod} (which holds for $\Phi^{[\ell]}_{n,\bm{p}}(H,\cdot)$), 
the reasoning around \cite[(6.25)-(6.26)]{cod} applies here verbatim, allowing us 
to set $\kappa_1=\kappa_1(H,t,\vep)$ and 
further restrict the minimization in \eqref{dfn:Phi-s} at $t+\vep$, to $X$ such that 
in addition $\BB_X$ intersects $\LL_F(\kappa_1)$ of \cite[(3.14)]{cod} for all $F \prec H$.
We proceed to set $\vep_0 \in [0,1]$ as in \cite[Sec. 6.3]{cod}, apart 
from replacing $\vep$ by $\vep \inf_n \{b_H^{(n)}\}$. As
$\BB_X$ satisfies \cite[(3.16)]{cod} for any $n \ge n_0$ and all $X \in \XX_n$
(thanks to our assumption that $n p^{2 \Delta_\star(H)} \gg 1$), we deduce 
from \cite[Prop. 3.7]{cod} that 
\[
\mathrm{hom}(H,X)\geq (t+\vep) b_H, \;\;
\BB_X \cap \LL_F(\kappa_1) \ne \emptyset,  \;\; \forall F \prec H \;\; 
\Longrightarrow \;\; \BB_X \subset \{ Z : \mathrm{hom}(H,Z) \ge t b_H \} \,.
\]
Considering the maximum of $\bP^{[\ell]}(\BB_X)$ over such $X$, we get from \eqref{eq:basic-lbd}
% and the preceding 
that for any $n \ge n_0$
\[
\log \bP^{[\ell]}(\mathrm{hom}(H,G^{[\ell]}_n) \geq t \, b_H ) \geq -\Phi^{[\ell]}_{n,\bm{p}}(H,t+\vep)-\eta_{n,p} \,.
\]
Dividing both sides by $a_{n,p}$ and taking $\vep=\vep_{n,p} \downarrow 0$ slowly enough, results with \eqref{eq:sbm2}.

	\section{Joint upper tails in Erd\H{o}s-R\'enyi graphs}\label{sec:joint}
	\subsection{Proof of \Cref{th:multprob}} By ignoring the requirements
	imposed on  $\mathrm{hom}(H_i,G_n)$ for $i>\sk$, we
	can and shall assume hereafter \abbr{wlog} 
	that $\sk'=\sk$, with $\Delta(H_i) = \Delta \ge 2$ for all $i \in [\sk']$. Further, in case
	$\Delta=2$ the range in \eqref{eq:p-range} be less stringent for cycles than for
	path (the only other connected graphs having $\Delta=2$). Thus, in {\bf Step 1}
    we take all $H_i=\sfC_{l_i}$ to be cycles and adapt the proof of
    \cite[(1.20)]{cod} for $p(n)$ determined by $l _o = \min_i \{l_i\}$, whereas in {\bf Step 2} 
    we consider the general case, adapting the proof of \cite[Thm. 1.1]{cod} with 
    $p(n)$ determined by $\Delta_\star := \max_i \{ \Delta_\star(H_i) \}$.
    % and $\verts_\star=\max \{ \verts(H_i) : \Delta_\star (H_i) = \Delta_\star \}$. 

\noindent
{\bf Step 1} 
Thanks to \cite[Lemma 4.1 \& Prop 4.2]{cod}, apart from $A_{G_n}$ in the $\bP_p$-negligible 
% \emph{single} 
event $\EE_0(\kappa_p)$ of \eqref{eq:GG-com}, 
we have that for suitable $\kappa_p \gg \log(1/p)$ and 
$k=k_n \gg (\log n)^{l_0/(l_0-2)}$ (see \cite[(7.3) \& (7.5)]{cod}), if 
$p(n)$ is in range \eqref{eq:p-range} for $H=\sfC_{l_o}$, then 
\[
 \|(A_{G_n})_{>k_n}\|_{S_l} \le \vep n p \,, \qquad \forall n \ge n_0, \quad \forall l \ge l_o \,.
 % $\vep_l := C' (n^{1/l-1/2} p^{-1/2} + \kappa^{1/2} R^{1/l-1/2})$ both go to zero as $n \to \infty$.
\]
We apply here \eqref{eq:ip2} with $\bm{p}=p \bm{1}$, 
$\R_+^\sk$-valued $h(\cdot)=\mathrm{hom}(\uu{H},\cdot)$,
$K=[0,1]^{n_\edges}$ and an exceptional set  which is 
determined by the multiple Schatten norms via
 \begin{equation*}
 \EE(\vep)
  \coloneqq \bigcup_{l \ge l_o} \{X \in \mathbb{B}_{\textrm{HS}}(n) :  \quad \|X_{>{k}}\|_{S_l} > \vep np  \} \,.
\end{equation*}
For the net $\bI := \Sigma \times \VV$ of cardinality $|\bI| = \exp(O(k n \log n))=\exp(o(a_{n,p}))$ from
\cite[Lemma 7.2]{cod} and $\delta' >0$ as in the proof of \cite[Thm. 7.1]{cod}, we enumerate 
over $y \in \bI$, taking the closed convex sets
\begin{equation*}
\bB_y(\varepsilon) \coloneqq \{ M(y)+W+Z  : \; W \in \bB_{\textrm{HS}}(\delta' n),\;  
Z \in \ZZ_{y}(\varepsilon) \, \} \cap  \bB_{\textrm{HS}}(n) \,, 
\end{equation*}
where for each $\vep > 0$, 
\begin{equation*}
\ZZ_{y}(\varepsilon) \coloneqq \{ Z \in \textrm{Sym}_n(\bR) : \textrm{Im}(Z)\subseteq \textrm{Ker}(M(y)),\;
 \max_{i=1}^\sk \{ \|Z\|_{S_{l_i}} \} \le \vep n p  \,\} \,.
\end{equation*}
		This is a suitable cover, since by \cite[Claim 7.4]{cod} any 
		$X \in \bB_{\textrm{HS}}(n) \cap \EE(\vep)^c$ must be in $\bB_{y(X)}(\vep)$,
		whereas \cite[Claim 7.5]{cod} yields the fluctuation bound 		
		\[
\max_{y \in \bI} \,  \sup_{X\in \bB_{y}(\vep)} \, 
		\| h(X)- h(M(y)) \|_\infty \leq \vep^{l_o} + o(1)  \,,
		\]
which as in  the proof of \cite[Thm. 7.1]{cod}, completes our proof of \eqref{eq:multprob} (for cycles).

		\noindent
		{\bf Step 2} By a union bound we deduce from \cite[(6.22)]{cod}  that  for
		any large enough $\kappa_1 \ge \kappa_1(\uu{H},\uu{t})$ and all $1 \gg p \gg (\frac{1}{n} \log n)^{1/(2\Delta_\star)}$,
		the event
		\begin{equation}
		\EE_{\uu{H}} (\kappa_1):= \{X\in \XX_n:   \max_{i \in [\sk]}  
	\max_{F \prec H_i} \{ \mathrm{hom}(F,X) \} > \kappa_1 \},
		\end{equation}
		has a negligible $\bP_p$-probability. Further, from 
		\cite[Prop. 3.7]{cod} and our choice of $\Delta_\star$, it follows that 
	%	as $p \to 0$, 
		for 
		%the $\R_+^k$-valued
		$h(\cdot)=\mathrm{hom}(\uu{H},\cdot)$,
		some $f_\star 
		% = f_\star(\uu{H}) :=1/ \max_i \{ \wt{C}_{H_i} \}
		>0$, any $\kappa_1 > 1$, all convex 
		$\BB \subset \EE_{\uu{H}}(\kappa_1)^c$ and $\vep \in [0,1]$,
		\begin{equation}\label{eq:bd-fluct}
		 \max_{X,Y \in \BB}\|X-Y\|_{\mathrm{op}} \leq \vep f_\star \kappa_1^{-1}
		  np^{\Delta_{\star}}  \qquad \Longrightarrow \qquad 
\sup_{X,Y \in \BB}  \|h(X) - h(Y) \|_\infty \le \vep \,.
		\end{equation}
		Fixing $\vep < 1$, we adapt the proof of \cite[Thm. 1.1]{cod}, applying
		again \eqref{eq:ip2}, now for
	    the $\bP_p$-negligible events  $\EE=\EE_0(\kappa_0) \cup \EE_{\uu{H}}(\kappa_1)$. 
	    We also take for $\BB_j=\CC_j$ the closed 
	    convex hull of sets from the net $\bI$ constructed in \cite[Prop. 3.4]{cod} for 
	    $\delta_o =\vep f_\star p^{\Delta_\star}/(4 C_\star \kappa_1)$
	    and $k_p = \lceil \kappa_0 (p^{\Delta}/\delta_o^2) \log (1/p) \rceil$, as in \cite[(6.14)-(6.15)]{cod}.
	    Thanks to \eqref{eq:p-range}, its 
	    cardinality is $|\bI| = \exp(O(k_p n \log n))=\exp(o(a_{n,p}))$, while combining  
	    \cite[(3.13)]{cod} with \eqref{eq:bd-fluct} yields maximal fluctuation 
	    $\vep$ of $h(\cdot)$ on each $\CC_j$. Thus, considering $\vep = \vep_{n,p} \downarrow 0$ 
	    slowly enough, concludes the proof of \Cref{th:multprob}.

	\subsection{Proof of \Cref{th:multvar}}
     We start with an asymptotically tight upper bound on the value of 
     $\phi^\sk_{n,p}(\uu{H},{\bf 1}+\uu{\delta})$,  analogously to \Cref{opt-reg}.
     \begin{propo}\label{opt-mult}
     For connected graphs $\{H_i, i \in [\sk]\}$, all of whom having maximal degree $\Delta \geq 2$.
     Fixing $\uu{\delta} \in \R_+^\sk$ and $x,y \ge 0$ such that \eqref{eq:xy-hom} holds simultaneously 
     for the pairs $(H_i,\delta_i)$, $i \in [\sk]$, 
     one has that for any $n^{-1/\Delta} \ll p=o(1)$, 
	\begin{equation}\label{eq:upd-phi-k}
		\limsup_{n \rightarrow \infty} \phi^\sk_{n,p}(\underline{H}, {\bf 1} + \underline{\delta}) \le x+ \frac{1}{2} y^2 \,.
		\end{equation}
	\end{propo}
	\begin{proof} By continuity, it suffices to consider $x,y>0$, for which 
	    our candidate be the 
	     weighted adjacency matrix $X_n^\star$ of \eqref{eq:clhub}, 
		with principal blocks of sizes $s_1  \sim x p^\Delta n$, $s-s_1$ for $s \sim y p^{\Delta/2} n$ and 
		$\wt{n}:= n-s$. Indeed, in the course of  proving \eqref{eq:prob-m-lbd-hub}, we have shown that 
		\begin{equation}\label{eq:Ip-val-xy}
		\frac{1}{2} I_p(X_n^\star) = (x + \frac{1}{2} y^2+o(1)) a_{n,p} 
		\end{equation}
		and consequently, it suffices to show that for any connected $H$ of maximal degree $\Delta \ge 2$, 
		\begin{equation}\label{eq:enough-hom}
		t(H,X_n^\star) \ge (y p^{\Delta/2})^{\verts(H)} +
		\sfP_{H^{\star}}(x) p^{\edges(H)}  + o(p^{\edges(H)}) \,.
		\end{equation}
		To this end, note that the contribution from having all vertices of $H$ 
	     in the middle block of $X_n^\star$ is $(s_2/n)^{\verts(H)} \sim  (y p^{\Delta/2})^{\verts(H)}$.
	     Proceeding to consider the contribution when no vertex of $H$ is within the middle block of $X_n^\star$, 
	     recall that $\edges(H_{S^c})=\edges(H) - \Delta |S|$ for any 
	     independent set $S$ of $H^\star$. Enumerating over the possible 
	     independent sets of $H^\star$,  the contribution to $t(H,X_n^\star)$
	     from having $S$ 
	     within the $s_1$-sized top principal block, is at least 
		\begin{align*}
		 \sum \limits_{S} \Big( \frac{s_1-1}{n} \Big)^{|S|} \Big(\frac{\wt{n}-1}{n} \Big)^{\verts(H)-|S|} q^{\edges(H_{S^c})} 
		&= (1+o(1))  \sum \limits_{S} x^{|S|} p^{\Delta |S|} q^{\edges(H_{S^c})}\\
	 &\ge (1+o(1)) \sfP_{H^{\star}}(x)p^{\edges(H)} \,.
		\end{align*}
	 This implies our claim \eqref{eq:enough-hom} and thereby completes the proof.
	\end{proof}
	\noindent
	Given \Cref{opt-mult}, similarly to {\bf Step 3} in proving 
	\Cref{thm:regvar}, it suffices to show that as $p \to 0$,
	\begin{align}\label{lbd:mult-var-cont}
		\phi^\sk_{p}(\underline{H}, \underline{\delta}) 
		&:= \frac{1}{2}  \inf \{ \frac{I_p(W)}{p^\Delta I_p(1)} : W \in \WW, 
		\; p^{-\edges(H_i)} t(H_i,W) \geq 1+\delta_i  \quad \forall i \in [\sk] \} \nonumber \\
		& \ge c(\uu{H},\uu{\delta}) - o(1)
		\end{align}
		(for $t(H,W)$ as in \eqref{dfn:t-graphon}). Since  
	   $x \mapsto I_p(x)$ is non-increasing on $[0,p]$, we may and will set $W=p+U$ in \eqref{lbd:mult-var-cont}, 
	   with $U \in [0,1-p]$. Then, following \cite{bglz} we define
	   for any $b \in [0,1]$ and such $U$, 
	    \begin{align*}
	    \BB_b = \BB_b(U) := \{ x \in [0,1] :  \int_0^1 U(x,y) dy \ge b \}, & \qquad \bar \BB_b := [0,1] \setminus \BB_b, \\
		x_b \coloneqq p^{-\Delta} \iint_{\BB_b \times \bar \BB_b} U^2(t,s) dt ds,& \qquad
		y^2_b \coloneqq p^{-\Delta} \iint_{\bar \BB_b \times \bar \BB_b} U^2(t,s) dt ds\,.
		\end{align*}
         Recall \eqref{eq:50} that for any $b \in [0,1]$,
		\begin{equation*}
		\frac{1}{2} I_p(p+U) \geq (1-o(1)) (x_b+\frac{1}{2} y_b^2) p^\Delta I_p(1).
		\end{equation*}		
		Further, from  \cite[(6.7)-(6.8)]{bglz} we know that if $I_p(p+U) = O(p^\Delta I_p(1))$
		and $p^{-\edges(H)} t(H,p+U) \geq 1+\delta$, then for any $b \to 0$ slowly enough in terms of $p$,
		\[ 
      \sfP_{H^{\star}}(x_b)+ y_b^{\verts(H)} \,
      \ind_{\{ H \;\;  \emph{is} \;\; \Delta\textrm{-}\mathrm{regular} \}} \ge 1+ \delta - o(1) \,.
      	 \]
	   The same choice of $b$ applies for 
	   multiple connected $H_i$ of maximal degree $\Delta \ge 2$, thereby bounding below 
	   $\phi^\sk_{p}(\cdot,\cdot)$ as in \eqref{lbd:mult-var-cont} and completing the proof of \Cref{th:multvar}.
	
	\subsection{Proof of \Cref{prop:jut}}
    Starting with the \abbr{ER} model, from Propositions \ref{th:multprob} and \ref{th:multvar} we get 
	\begin{equation}\label{eq:ubd-mult}
	\limsup_{n \rightarrow \infty} a_{n,p}^{-1} \log\bP_p(\mathrm{hom}(\underline{H},G_n) \geq {\bf 1}+\underline{\delta}) 
	\leq - c(\pi_\sk(\underline{H}),\pi_\sk(\underline{\delta})).
	\end{equation}
	Pittel's inequality \eqref{eq:ubd-unif-Gn} with
	 $\AA_n = \{ \mathrm{hom}(\uu{H},G_n) \geq {\bf 1} + \uu{\delta} \}$, yields for any $n$ and $m$, 
\begin{equation}\label{eq:ubd-unif-mult}
	\bP^{(m)} (\mathrm{hom}(\uu{H},G_n) \geq {\bf 1} + \uu{\delta})   
	\le 3 \sqrt{m} \,  \bP_p(\mathrm{hom}(\uu{H},G_n) \geq {\bf 1} + \uu{\delta}) \,,
	\end{equation}	
hence in view of \eqref{eq:p-range}, the upper bound \eqref{eq:ubd-mult} applies  also 
	for the law $\bP^{(m)}(\cdot)$.  Turning to the complementary lower bound, note that the probability
	of $\{\mathrm{hom}(\uu{H},G_n) \geq {\bf 1}+\uu{\delta} \}$ under both 
	$\bP_p(\cdot)$ and $\bP^{(m)}(\cdot)$ is at least
\[	
\bP_p(\{\mathrm{hom}(\uu{H},G_n) \geq {\bf 1}+\uu{\delta} \} \cap \KK^{(m)}_n) 	\,,
\]	
and thereby it suffices to lower bound the rate of decay of the latter. That is, to show  that 
per fixed $\uu{\delta} > \uu{\delta'} \in \R_+^{\sk'}$ and $x,y>0$ which 
satisfy \eqref{eq:xy-hom} simultaneously for $(H_i,\delta_i)$, $i \in [\sk]$, one has the bound
\begin{align}
\liminf_{n \to \infty} a_{n,p}^{-1} \log \bP_p(\{\mathrm{hom}(\uu{H},G_n) \geq {\bf 1}+\uu{\delta'}\} \cap \KK^{(m)}_n) 
&\ge - x - \frac{1}{2} y^2 \,.
\label{eq:lbd-unif-mult}
\end{align}	
The proof of \eqref{eq:lbd-unif-mult} proceeds precisely as the derivation of  
\eqref{eq:prob-m-lbd-hub}, by first making a change of the measure to the planted 
\abbr{er}-model $\bP_\star$ that corresponds to $X_n^\star$ of \eqref{eq:clhub}, after
which it remains only to show that \eqref{eq:star-lln} holds simultaneously for all
$(H_i,\delta'_i)$, $i \le \sk'$.  In case $i \le \sk$ the same argument as in the proof of
\eqref{eq:star-lln} applies here, thanks to our assumption that \eqref{eq:xy-hom} 
holds for $H_i$ and some $\delta_i>\delta'_i$. Next, fixing $i  \in (\sk,\sk']$,
for $H=H_i$ let $S^c$ be any maximal subset of $V(H)$, such that  $\Delta(H_{S^c}) \le \Delta$.
Clearly $S \ne \emptyset$ since $\Delta(H_i) > \Delta$, while the 
maximality of $S^c$ implies that $\edges(H) \ge \edges(H_{S^c}) + (\Delta+1) |S|$.
Consequently,
the contribution under $\bP_\star$ to $\mathrm{hom}(H,G_n)$
from homomorphisms with $S$ in the hub of size $s_1$ of $X^\star_n$ 
and $S^c$ in its ${\bf q}$-block of size $\wt{n}$, is at least
$\mathrm{hom}(H_{S^c},G_{\wt{n}})$ times 
\[
\Big(\frac{s_1-1}{n}\Big)^{|S|} \Big(\frac{\wt{n}-1}{n}\Big)^{|S^c|} q^{\edges(H_{S^c})} p^{-\edges(H)}
\ge x^{\Delta |S|} p^{-\Delta} (1-o(1))\,.
\] 
Since the latter (non-random) factor diverges for any $x>0$ fixed and $p=p(n) \to 0$, the required bound 
\eqref{eq:star-lln} holds for $H=H_i$ and \emph{any fixed} $\delta'<\infty$, provided  
\begin{equation}\label{eq:lower-tail-lbd}
\liminf_{n \to \infty} a_{n,p}^{-1} \log \bP_q(\mathrm{hom}(H_{S^c},G_{\wt{n}}) < 1-\vep) < 0, \quad
\forall \vep >0, \quad 1 \gg p(n) \gg n^{-1/\Delta} \,.
\end{equation}
By construction  
$\Delta(H_{S^c}) \le \Delta$, in which case we have already proved \eqref{eq:lower-tail-lbd} 
(in the course of proving \Cref{th:correg+unif}). In conclusion,  
\eqref{eq:star-lln} holds for all $H_i$, $i \in [\sk']$, hence \eqref{eq:lbd-unif-mult} holds as well, 
completing the proof of the theorem.

	\section*{Acknowledgment}
We thank Nick Cook, Anita Liebenau and Subhabrata Sen for helpful discussions. 
We further thank Sourav Chatterjee, Ben Gunby, Eyal Lubetzky, Andrzej Rucinski, Yufei Zhao and the anonymous referee,
for their feedback on our first drafts, and Wojciech Samotij for suggesting the improvement via the use of \cite{samotij}
(cf. \Cref{rem:samot}). 
\\
The research of A.D. is partially supported by NSF grant DMS-16130.

	%d
	\bibliographystyle{plain}
	\bibliography{abc}

\end{document}